\documentclass[10pt]{article}
\usepackage[textwidth=15cm, textheight=22cm]{geometry}
\usepackage{amsfonts}
\usepackage{amsmath}
\usepackage{amssymb}
\usepackage{amsthm}
\usepackage{enumitem}
\usepackage{graphicx}
\usepackage{tikz}
\usetikzlibrary{arrows.meta,shapes,calc}
\usetikzlibrary{positioning,shapes,arrows}

\usepackage{pgfplots}
\pgfplotsset{compat=1.12}
\usetikzlibrary{plotmarks}
\usepackage[ngerman, english]{babel}
\usepackage{babelbib}
\usepackage[latin1]{inputenc}

\usepackage{nicefrac}

\usepackage{subfig}
\usepackage{epstopdf}

\usepackage{bbm}

\usepackage{url}

\usepackage{epstopdf}
\usepackage[dvips]{epsfig}

\usepackage{graphicx}%

\newcommand{\FF}{\mathcal{F}}
\newcommand{\GG}{\mathcal{G}}

\newcommand{\UU}{\mathcal{U}}

\newcommand{\E}{\mathbb{E}}

\newcommand{\Prob}{\mathbb{P}}

\newcommand{\R}{\mathbb{R}}

\newcommand{\hatt}{\hat{t}}

\newcommand{\tm}{\tilde{\mu}}

\newcommand{\ttt}{\tilde{t}}

\newcommand{\invLa}{\nicefrac{1}{\lambda}}

\newcommand{\defgl}{\mathrel{\mathop:}=}

\theoremstyle{plain}
\newtheorem{thm}{Theorem}[section]

\newtheorem{prop}[thm]{Proposition}

\theoremstyle{definition}

\theoremstyle{remark}
\newtheorem*{rem}{Remark}

\begin{document}
\title{Optimal control of electricity input given an uncertain demand}
\author{Simone G\"{o}ttlich\footnotemark[1], \; Ralf Korn\footnotemark[7], \; Kerstin Lux\footnotemark[1]}
\footnotetext[1]{University of Mannheim, Department of Mathematics, 68131 Mannheim, Germany (goettlich@uni-mannheim.de, klux@mail.uni-mannheim.de).}
\footnotetext[7]{TU Kaiserslautern, Department of Mathematics, 67653 Kaiserslautern, Germany (korn@mathematik.uni-kl.de) and Fraunhofer ITWM, Department of Financial Mathamtics, 67663 Kaiserslautern.}
\maketitle

\begin{abstract}
We consider the problem of determining an optimal strategy for electricity injection that faces an uncertain power demand stream. This demand stream is modeled via an Ornstein-Uhlenbeck process with an additional jump component, whereas the power flow is represented by the linear transport equation. We analytically determine the optimal amount of power supply for different levels of available information and compare the results to each other.
For numerical purposes, we reformulate the original problem in terms of the cost function such that classical optimization solvers can be directly applied. The computational results are illustrated for different scenarios.

\end{abstract}

\section{Introduction} \label{sec:Intro}
With the liberalization of the electricity markets in Europe, the modeling of energy prices has become a very active field of research and benefited from a combination of methods adapted from financial mathematics on the one hand and the way prices are formed on energy markets on the other hand. One popular approach is to use so-called structural models. They are based on modeling the electricity demand $Y_t$ at time $t$ by (variants of) Ornstein-Uhlenbeck processes and obtaining the electricity price $S_t$ (for a suitable unit of electricity such as MW) as a deterministic function $f(Y_t)$ of the demand (see \cite{Barlow.2002} as the first and influential source for this approach). Those kind of models have been further developed in a series of papers, see e.g. \cite{Aid.2009, Kiesel.2009, LuciaSchwartz.2002, SchwartzSmith.2000, Wagner.2014}, or the monograph \cite{Benth.2008}.

However, this strand of literature more or less takes the decision problem on the scheduling of electricity input and its distribution to the customer as given. It is mainly concerned with the mechanism of the price agreement. We use a somewhat orthogonal approach and assume that the price decision has already been made by the electricity provider and that the main challenge consists in the actual electricity injection to satisfy the demand as good as possible. 

The problem that we consider contains two major challenges, the modeling of all ingredients involved and the way we can control the electricity input. To deal with the first aspect, we are facing three modeling tasks:
the modeling of the electricity demand (given a forecast for the demand over a specified time span, e.g. one day, one week), the modeling of the electricity input (again, given a forecast of the demand), and the modeling of electricity transmission to the customer. 

To decide on the electricity input given the above modeling ingredients, our objective criterion is the minimization of the quadratic deviation of produced power from the actual demand. There, the main challenge is the modeling of the control possibilities. We again consider and compare three approaches: 
\begin{itemize}
\item the idealized situation where the provider monitors the demand continuously in time and uses a corresponding feedback control of the power production as one extreme,
\item the use of a discrete-time demand monitoring on a given time grid with control actions based on this information, and
\item the ignorance of the actually evolving demand where one uses the forecasted demand as the only information basis for the production.  
\end{itemize}

The main contribution of this paper is therefore to provide
a complete modeling setup for the controlled input of electricity, the realistic consumer demand and the electricity transmission. Furthermore, we present an explicitly solved idealized stochastic control problem for the optimal power input and a comparison to practical control schemes that highlights their quality. We also point out aspects for future research on extending the model and control setup to larger electricity systems.  

The outline of the paper is as follows: We will set up the modeling ingredients in Section \ref{sec:ProblemSetting} where, in addition to the standard Ornstein-Uhlenbeck based modeling, we also consider the use of a jump-diffusion Ornstein-Uhlenbeck process. The announced different control strategy approaches are presented in Section \ref{sec:TheoreticalOCapproaches}. A numerical study in Section \ref{sec:Numerics} highlights the different demand modeling and control approaches. A conclusion sets the stage for future research.

\section{The stochastic optimal control problem} \label{sec:ProblemSetting}
In this section, we describe the choice of an optimal power supply depending on a stochastic power demand as a stochastic optimal control problem. To do so, we model the power demand as a stochastic process and come up with a suitable control problem restricted by the energy transport equation.
Optimal control strategies for the power supply problem subject to given deterministic demands have been considered in \cite{Goettlich.2016,Goettlich2018}. 

\subsection{Problem description} \label{subsec:problemSet}
We first assume the simplest possible form of an electricity system, i.e the flow on a single line. 
Physically, this means, at $x=0$, the power is injected and leaves the system at $x=1$.
We also consider a finite time interval $[0,T]$. Within this period, the power inflow $u(t)\in L^2$ at $x=0$ and the externally given customers' demand $Y_t$ located at $x=1$ are the quantities of interest that need to be matched in an optimal way. 
For simplicity, we assume that the dynamics of the electricity $z(x,t)$ are governed by the linear transport (or advection) equation with constant transport velocity $\lambda>0$ and the following initial and inflow conditions:
\begin{align}
	z_t + \lambda z_x &= 0, \quad x \in (0,1),\ t \in [0,T] \notag \\
	z(x,0) &= z_0(x), \quad z(0,t) = u(t).\label{eq:transportDyn}
\end{align}
In the following, $y(t)=z(1,t)$ denotes the outflow of the system that should be adjusted to the customers' demand $Y_t$.
As we have $u(t) = y(t+\nicefrac{1}{\lambda})$, the output directly results from the choice of the inflow $u(t)$. 

The arising constrained stochastic optimal control problem is then given by
\begin{align}
	\min_{u(t), t \in [0,T-\nicefrac{1}{\lambda}], u \in L^2} &  \E \left[\int_{\nicefrac{1}{\lambda}}^{T} h(Y_s,y(s))ds\right] \ \text{subject to} \ \eqref{eq:transportDyn} , \label{eq:OCrough}
\end{align}
where $\nicefrac{1}{\lambda}$ is the transportation time, $h:\R\times\R \rightarrow \R$ an appropriately chosen loss function, and $\left(Y_t\right)_{t \in [0,T]}$ a given stochastic demand process. 

Next, we present suitable stochastic processes for the modeling of the electricity demand and a reasonable choice for the loss function $h$.

\subsection{Modeling of demand} \label{subsec:Demand}
The actual demand $Y_t$ at the end of the power line, i.e. at $x=1$, is a stochastic time-dependent quantity due to uncertainty about the height and the timing of the customers' demand. However, there are various indicators such as historic demands, current demand, and the actual time, as well as weather and demand forecasts that hint at the use of some specific types of stochastic processes $\left(Y_t\right)_{t \in [0,T]}$, which we are going to discuss below.

\subsubsection{Ornstein-Uhlenbeck process (OUP)}
Given the availability of historical data on electricity demands and also short-term forecasts for each time of the day, it is reasonable to assume the existence of a deterministic process $\mu(t)$ around which the actual demand fluctuates. Indeed, we can identify $\mu(t)$ with the predicted demand at time $t$. Assuming that the remaining uncertainty (\textit{the fluctuation around $\mu(t)$}) is level- and time-independent, the Ornstein-Uhlenbeck process (OUP) is the natural candidate to model the demand. It is given by the stochastic differential equation (SDE)
\begin{align}
	dY_t=\kappa\left(\mu(t)-Y_t\right)dt + \sigma dW_t,\ \quad Y_{0}=y_0, \label{eq:OUP}
\end{align}
where $W_t$ is a one-dimensional Brownian motion, $\sigma,\ \kappa$ are positive constants, and $y_0$ describes the initial demand. Whenever the current demand is higher (lower) than $\mu(t)$, the negative (positive) drift term implies a reversion towards the mean demand level $\mu(t)$. The speed of this mean reversion is described by $\kappa$ and the intensity of demand fluctuations by $\sigma$.

The way the mean reversion speed $\kappa$ and intensity of demand fluctuations $\sigma$ affect the evolution of the demand is shown exemplarily for a constant mean reversion level $\mu(t) \equiv \mu = 10$ in Figure \ref{fig:OUP_fixLevel2_kappa1_sigma1_T8_MC10P3}. Each time, we plot three realizations of the solution of \eqref{eq:OUP} corresponding to the initial values $y_0=6, y_0=9$, and $y_0=14$. The intensity of demand fluctuations is given by $\sigma=1$ in the first row and $\sigma=2$ in the second row. This results in larger fluctuations around the mean demand level $\mu$ (depicted by the black solid line).
From column one to column two, we increase the speed of mean reversion from $\kappa=1$ to $\kappa=3$. As a consequence, it takes the demand process less time to return to the mean demand level $\mu$ whenever it is away from it.
\begin{figure}[h!]
	\subfloat[\ $\kappa=1$, $\sigma=1$]{\includegraphics[width=0.49\textwidth]{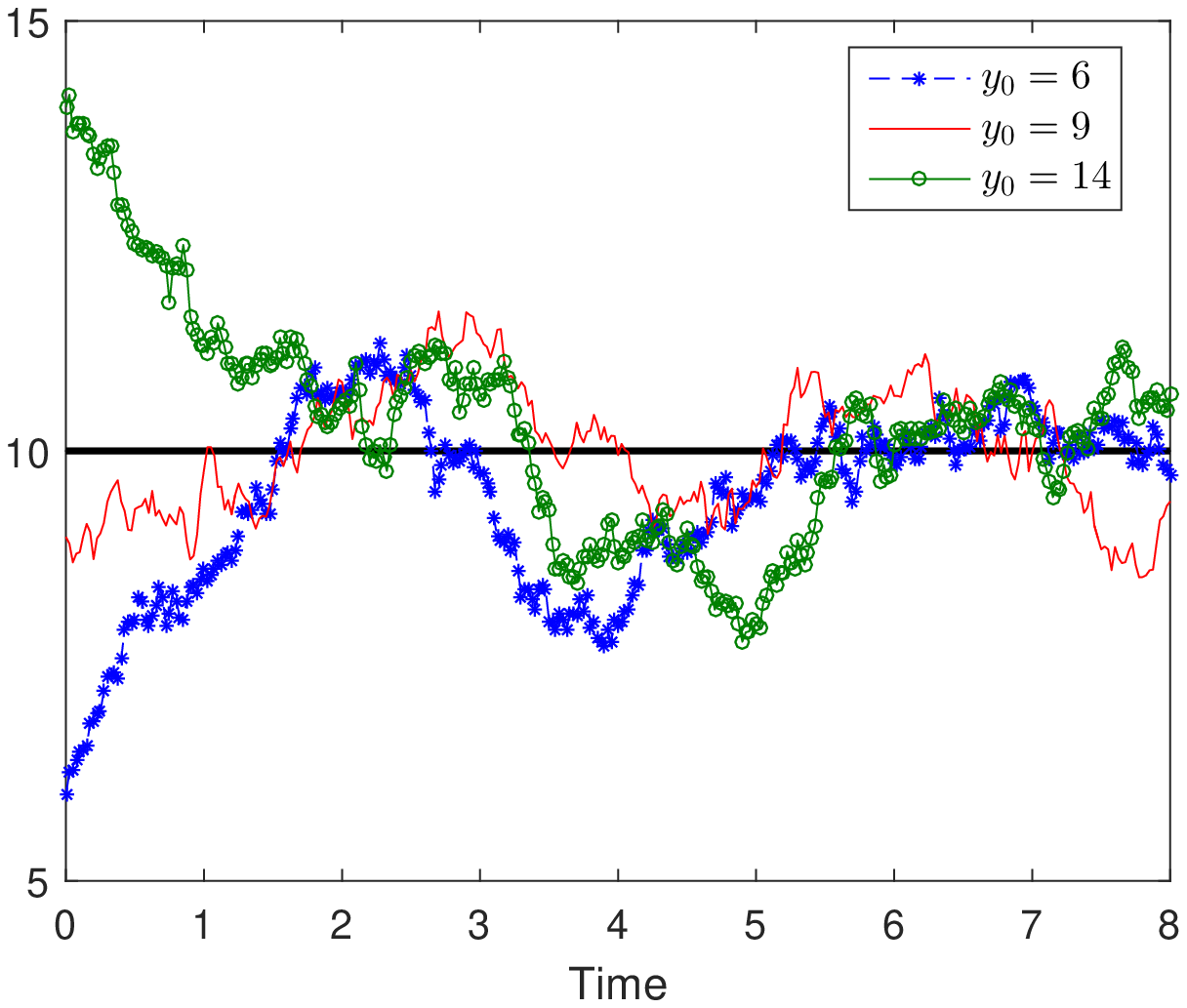}}\hfill
	\subfloat[\ $\kappa=3$, $\sigma=1$]{\includegraphics[width=0.49\textwidth]{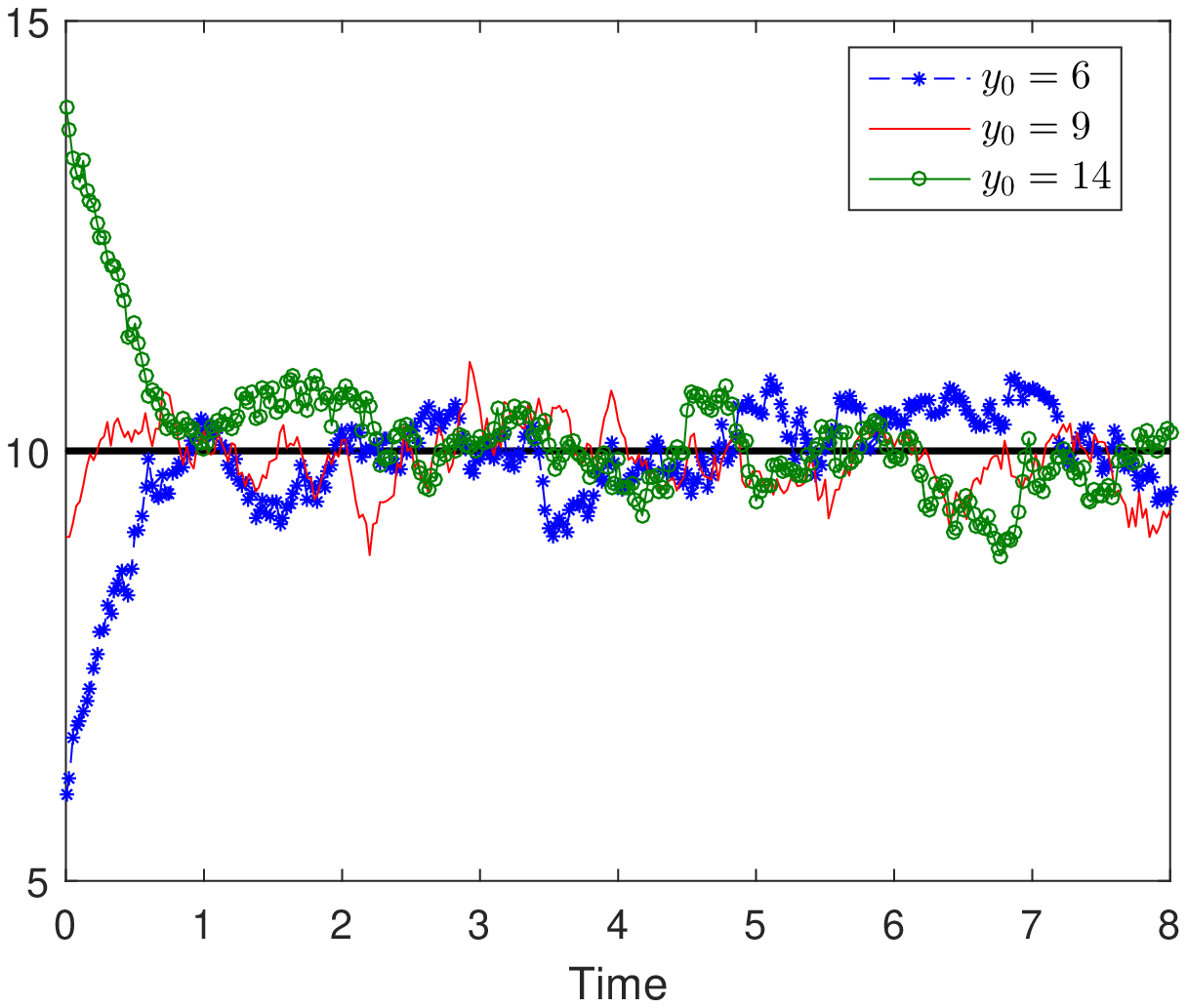}}
    
	\subfloat[\ $\kappa=1$, $\sigma=2$]{\includegraphics[width=0.49\textwidth]{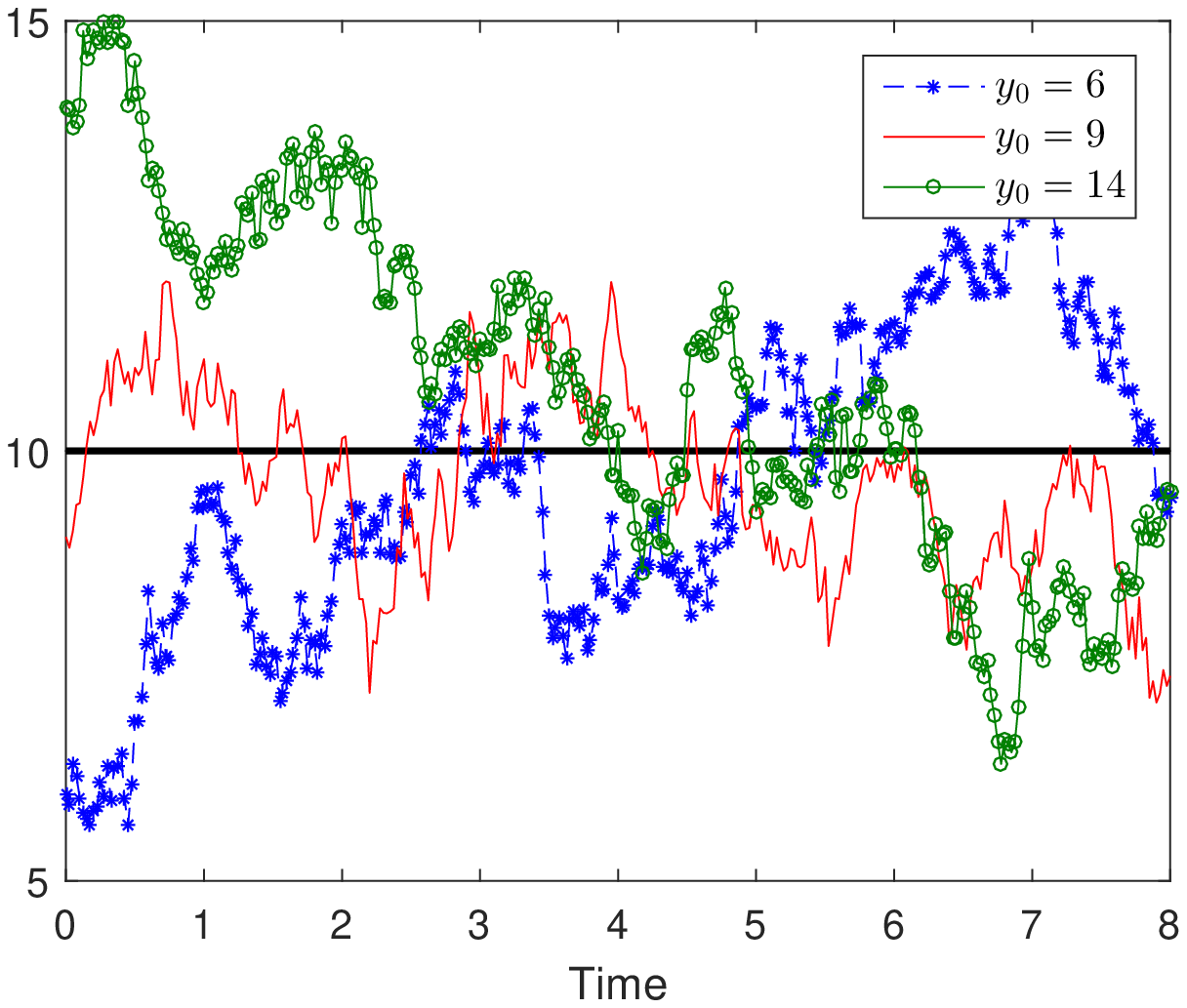}}
	\subfloat[\ $\kappa=3$, $\sigma=2$]{\includegraphics[width=0.49\textwidth]{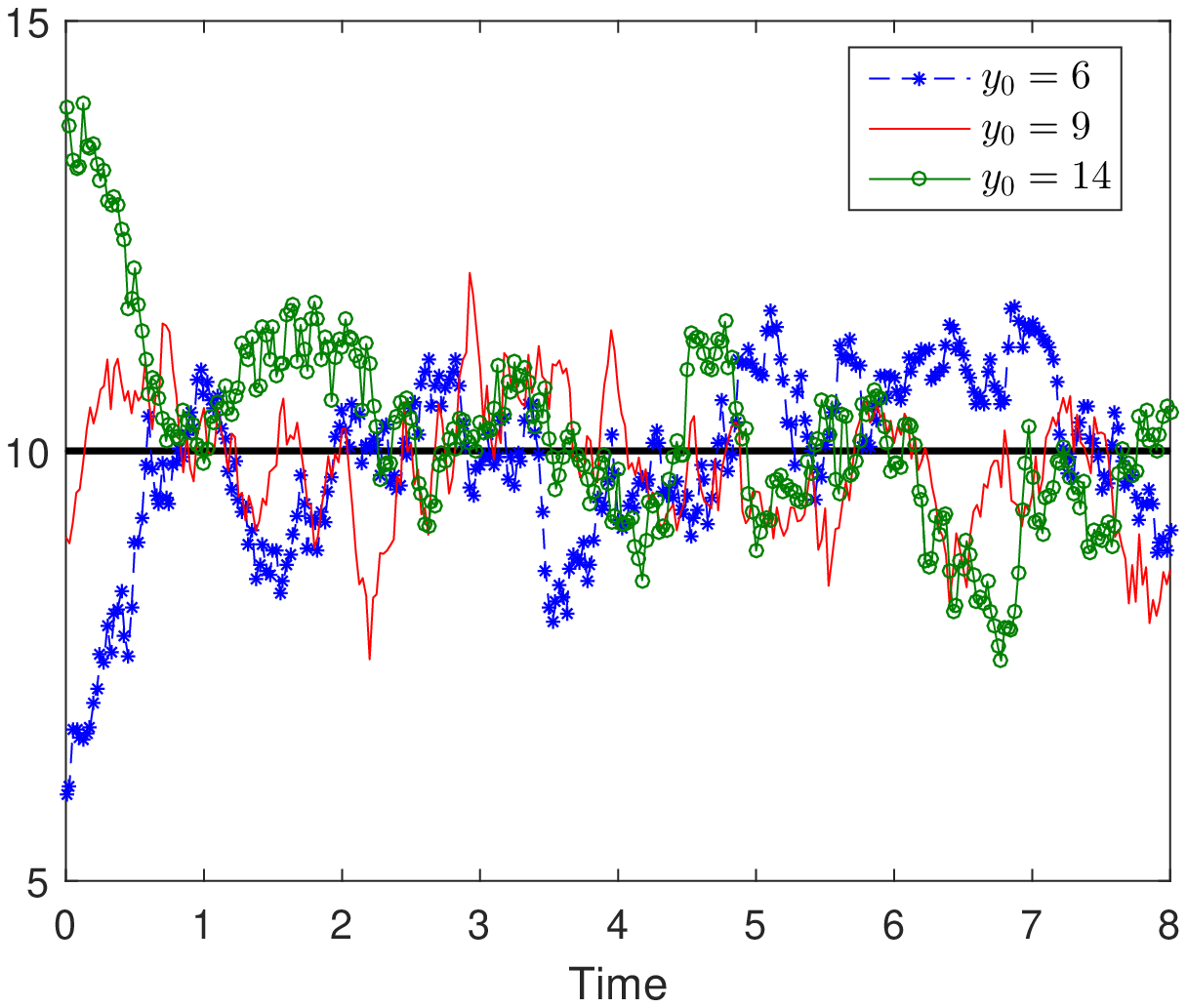}}
	
	\caption{Influence of mean reversion speed $\kappa$ and intensity of demand fluctuations $\sigma$ on the demand}
	
	\label{fig:OUP_fixLevel2_kappa1_sigma1_T8_MC10P3}
\end{figure}

The SDE \eqref{eq:OUP} has an explicit solution given by 
\begin{equation}
Y_t = y_0 e^{-\kappa t} + \kappa \int_0^t{\mu(s) e^{-\kappa(t-s)}ds} + \sigma \int_0^t{e^{-\kappa(t-s)}dW_s} . \label{eq:solOUP}
\end{equation}
Even more, one can infer the distribution of $Y_t$ directly from its explicit form as being normally distributed according to
\begin{equation}
Y_t \sim N\left( {y_0e^{ - \kappa t}  + \kappa \int\limits_0^t {e^{ - \kappa \left( {t - s} \right)} \mu \left( s \right)ds} ,\,\,\sigma ^2 \int\limits_0^t {e^{ - 2\kappa \left( {t - s} \right)}ds} \,} \right).
\end{equation}
The obvious drawback from this solution and its distribution is that the OUP can attain negative values. However, for reasonable parameters, this is not of great interest. To see this, note that in the special case of a constant positive demand $\mu(t) \equiv \mu>0$, we have 
\begin{equation}
Y_t \sim \mathcal{N}\left(\mu + (y_0-\mu)e^{-\kappa t}, \frac{\sigma^2}{2 \kappa}(1-e^{-2\kappa t})\right)\ .
\end{equation}
Thus, as $\mu$ is positive, the probability for a negative value of $Y_t$ gets negligible if the mean reversion $\kappa$ is large compared to $\sigma^2$. This should be the case if the forecast of the demand $\mu(t)$ is reliable. 

The tractability of the process, the desirable property of mean reversion, and the possibility to use forecasts for the demand have made the OUP a popular model for power demand in the electricity price literature, see e.g. \cite{Benth.2008}.

\subsubsection{More general OUP: adding jump components}
To allow for a more general demand behavior, one can replace the Brownian motion $W_t$ in equation \eqref{eq:OUP} by a general L\'{e}vy process (see e.g. \cite{Applebaum.2009} for a survey on L\'{e}vy processes) or can add a jump martingale component. This is particularly appealing as jumps in the electricity demand can occur e.g.\ due to an unexpected start or end of an industry process. 

As a consequence, we include a special type of L\'{e}vy process, the compound Poisson process, into the demand process so that we obtain a jump diffusion process (JDP) version of the OUP \eqref{eq:OUP} of the following form:
\begin{align}
	dY_t = \kappa\left(\mu(t)-Y_t\right)dt + \sigma dW_t + \gamma_t dN_t,\ \quad Y_{0}=y_0. \label{eq:nonCompJDP}
\end{align}
Here, $W_t$ is again a one-dimensional Brownian motion, and $N_t$ is a homogeneous Poisson process with intensity $\nu > 0$, independent of $W_t$. Further, $\gamma_t$ is assumed to be a time-homogeneous stochastic process, which is independent of both the Poisson process and the Brownian motion.
Note that $N_t$ admits only jumps of height $1$. The time between two jumps is exponentially distributed with parameter $\nu$. Both facts together yield that the increments $N_t - N_s$ between two points in time $s<t$ are Poisson-distributed with parameter $\nu(t-s)$. While $N_t$ thus determines if and when a jump in the demand process occurs, $\gamma_t$ models the actual jump height at time $t$ given a jump occurs. For example, one can think of a jump height distribution as a normal or a log-normal distribution. 

As before, the evolution of the demand is shown exemplarily for a constant mean reversion level $\mu(t) \equiv \mu = 10$ (depicted by the black horizontal line) in Figure \ref{fig:JDP_fixLevel10_kappa1_sigma1_T8_MC10P3} including jumps now. Again, we plot three realizations of the solution of \eqref{eq:nonCompJDP} corresponding to the initial values $y_0=6, y_0=9$, and $y_0=14$, based on the same realizations of jump times and Brownian increments to allow for a visualization of the identical jump times by black vertical lines.
After a jump, for a higher mean reversion speed, the process returns faster to its mean reversion level and the amplitude of values around the mean demand level attained by the process is lower.
\begin{figure}[h!]
\hspace{0.5cm}
	\subfloat[\ $\kappa=1$, $\sigma=2$]{\includegraphics[width=0.42\textwidth]{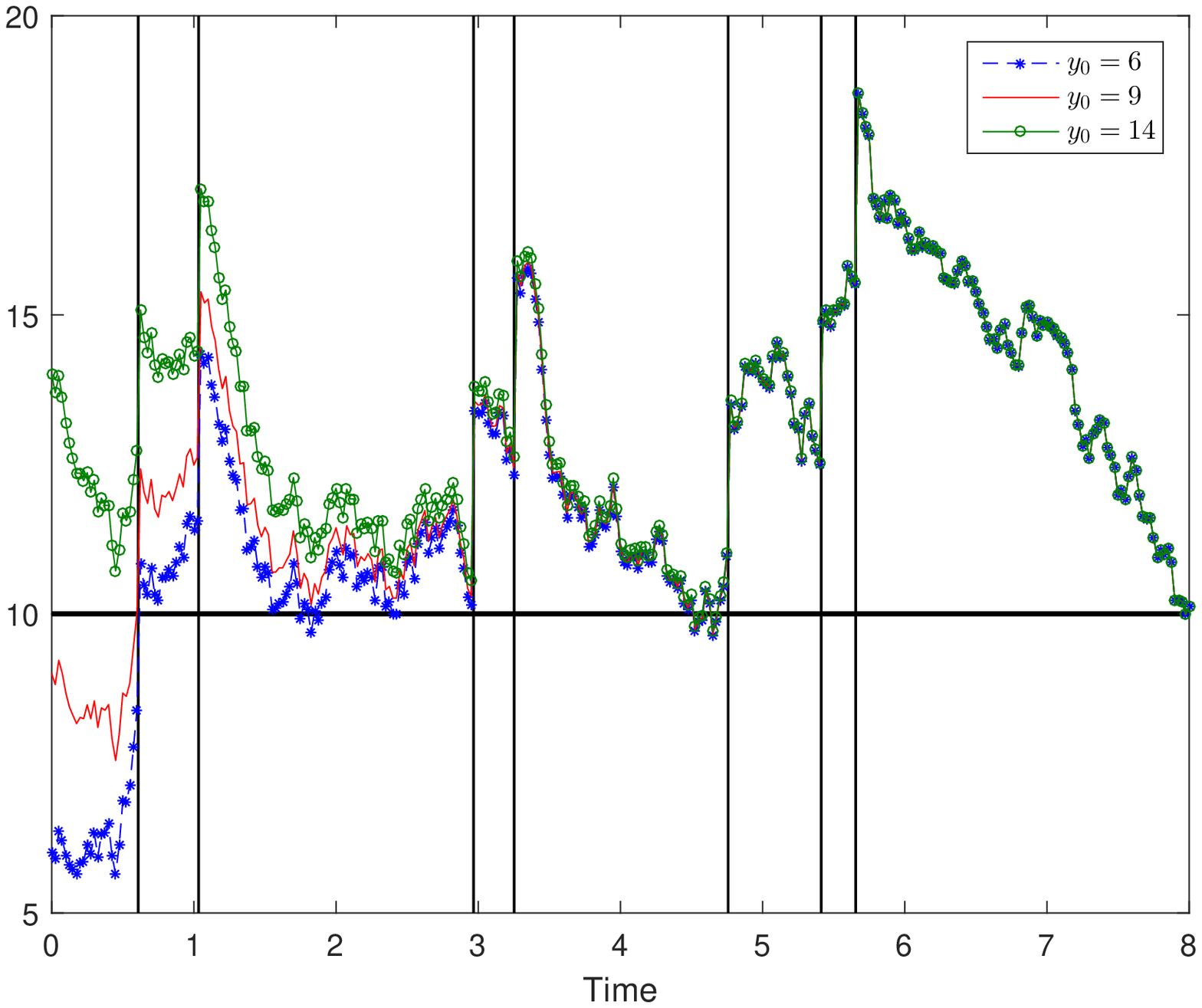}}\hspace{1cm}
	\subfloat[\ $\kappa=3$, $\sigma=2$]{\includegraphics[width=0.42\textwidth]{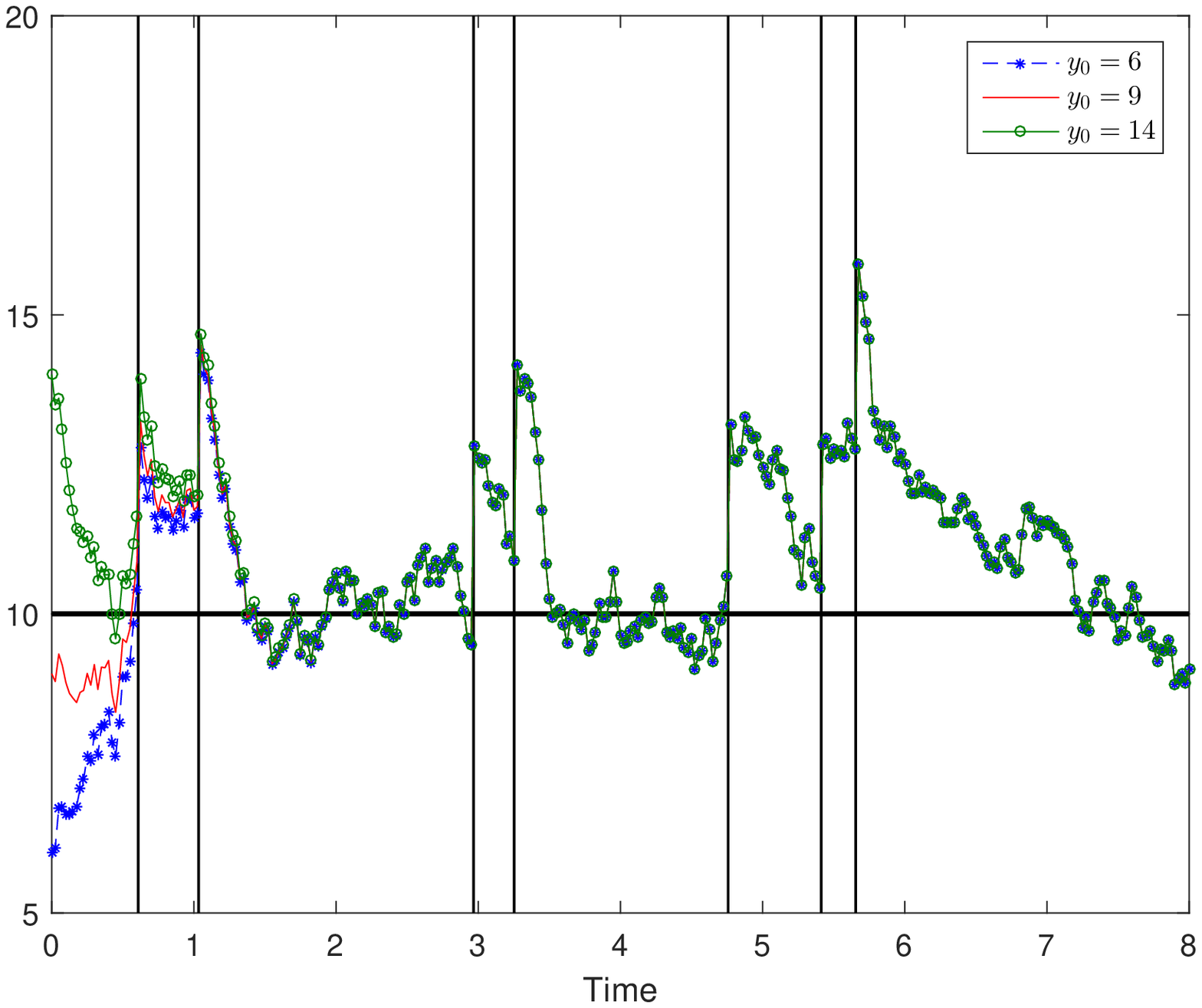}}

	\caption{Demand behavior in the presence of jumps for different mean reversion speeds}
	
	\label{fig:JDP_fixLevel10_kappa1_sigma1_T8_MC10P3}
\end{figure}

Also the JDP equation \eqref{eq:nonCompJDP} has a unique explicit solution.
Indeed, it can directly be verified by the It\^{o}-formula for jump diffusions that we have
\begin{align}
	Y_t =& y_0e^{-\kappa t} + \kappa \int_{0}^{t} \mu(s)e^{-\kappa (t-s)} ds + \sigma \int_{0}^{t} e^{-\kappa (t-s)} dW_s + \sum_{i=1}^{N_t} \gamma_{t_i} e^{-\kappa (t-t_i)}, \label{eq:solJDP}
\end{align}
where the random times $t_i$ are the jump times of the Poisson process $N_t$.
Although, $\gamma_t dN_t$ is a formal abbreviation of an integral, this is actually only a sum composed of a random number of terms, i.e., we have
\begin{equation}
\int_0^t{\gamma_s dN_s} = \sum_{i=1}^{N_t}{\gamma_{t_i}}
\end{equation} 
where $0 < t_1 < ... < t_{N_t} \le t$
denote the jump times of the Poisson process until time $t$. The sum on the right-hand side is called a compound Poisson process. In the simple case of a constant jump height $\gamma > 0$, the compound Poisson process is simply $\gamma$ times a Poisson process. 
To obtain a form of the differential representation of the JDP that allows for a similar interpretation as in the pure diffusion case, we have to transform the jump part into a martingale. This is compensated by a suitable adjustment of the drift term. 

With the introduction of the compensated Poisson integral given by
\[
\gamma_t \widetilde{dN_t} \defgl \gamma_t dN_t - \nu \bar{\gamma} dt,\
\]
and the notation $\bar{\gamma}=\E(\gamma_t)=\E(\gamma)$, this integral is a martingale. This then allows for the desired equivalent formulation of the jump diffusion representation \eqref{eq:nonCompJDP} as
\begin{align}
dY_t = \kappa\left(\mu(t)-Y_t+\frac{\bar{\gamma}\nu}{\kappa}\right)dt + \sigma dW_t + \gamma_t \widetilde{dN_t}\ .\label{eq:jumpSDE}
\end{align}

With this representation, the parameter $\kappa$ regains its interpretation as the speed of mean reversion to the mean demand level. Note, however, that the presence of the compensated jump process has changed the mean reversion level to 
\[
\tm(t) = \mu(t) + \frac{\bar{\gamma} \nu}{\kappa} \ .
\]
By the OUP \eqref{eq:OUP} and the jump diffusion process JDP \eqref{eq:nonCompJDP}, we have two possibi\-lities to model the stochastic electricity demand.

\subsection{Choice of the objective function} \label{subsec: OF}
To complete the formulation of the constrained stochastic optimal control problem \eqref{eq:OCrough}, we specify the loss function $h$. 
As both excess supply as well as undersupply should be penalized, it seems reasonable to measure the quality of the control strategy that results in the electricity output in terms of the quadratic deviation between realized demand $Y_t$ and output $y(t)$ over time. In \eqref{eq:OCrough}, this leads to the objective function given by
\begin{align}
	OF(Y_s,y(s)) = \int_{\nicefrac{1}{\lambda}}^{T} \E\left[(Y_s - y(s))^2 \right] ds. \label{eq:OF}
\end{align}
As this choice is more suitable to figure out a good control strategy than to actually describe exact gains or losses (they depend on pricing issues and contractual aspects that we do not face in this paper), we only consider the choice \eqref{eq:OF}. 
Furthermore, the objective function \eqref{eq:OF} allows for an explicit calculation of the optimal control strategy (see Subsection \ref{subsec:OCdiffset}).

Note that the controller has (at most) information about the demand $Y_s, s \leq t$ to determine the optimal output $y(t+\frac{1}{\lambda})$ which then faces the demand $1/\lambda$ units of time later. In measurability terms, the control input $u(t)$ should be predictable with respect to the $\sigma$-algebra $\mathcal{F}_{t}$, $ t \in [0, T-\frac{1}{\lambda}]$. Thus, a perfect match between the future demand and the one about which the controller can decide now would require knowledge of future information. Hence, a perfect match between future demand and actual control is not possible.

Even more, considering controls $u(t)$ that are based on full information about the actual demand process up to time $t$ at all times are hardly realistic for practical purposes. Therefore, three types of controls  -- two of them admitting more realistic choices of measurability assumptions -- are presented and compared in Subsection \ref{subsec:diffSet}.

Summarizing, we state the complete constrained stochastic optimal control (SOC) problem as follows:
\begin{subequations}
\begin{align}
	\min_{u(t), t \in [0,T-\nicefrac{1}{\lambda}], u \in L^2} & OF\left(Y_s,y(s)\right) \label{eq:SOC-Prob}\\
	z_t + \lambda z_x &= 0, \quad x \in (0,1),\ t \in [0,T] \label{eq:transportDynInOptProb}\\
	z(x,0) &= z_0(x), \quad
	z(0,t) = u(t). \label{eq:control} \\
	dY_t &= \kappa\left(\mu(t)-Y_t\right)dt + \sigma dW_t + \gamma_t dN_t,\ \quad Y_{0}=y_0. \label{eq:JDPinOptProb}
\end{align}
\end{subequations}
The formulation of the stochastic optimal control problem for an OUP-type demand is obtained by setting $\gamma_t \equiv 0$.
\section{Optimal control strategies: different information levels} \label{sec:TheoreticalOCapproaches}
In this section, we focus on the solution of the SOC problem \eqref{eq:SOC-Prob}-\eqref{eq:JDPinOptProb}.
A common approach to tackle such a control problem might be a Fokker-Planck based control framework. This means, the distribution properties of the stochastic process are represented by the corresponding Fokker-Planck equation describing the evolution of the probability density of the process over time (see \ e.\ g.\ \cite{Borzi.2013, Breitenbach.2018, Gaviraghi.2017a, Roy.2018}, and more recently  \cite{Annunziato.2018}).

However, the SOC problem under consideration benefits from the explicit relation between inflow and outflow, i.e. $u(t) = y(t+\nicefrac{1}{\lambda})$, and the fact that the stochastic process itself does not contain the control variable. Therefore, there is no need to use a Fokker-Planck reformulation since the control strategy can be computed straightforward.

We first start with an idealized setting, where we aim at determining a best-possible optimal control $u(t) \in L^2$ based on the stochastic characteristics of the demand process.

\subsection{Conditional expectation as the general solution} \label{subsec:diffSet}
We first state a general projection result yielding that the conditional expectation is the $L^2$-optimal approximation of a random variable. Although, this is a well-known result, we give the proof for completeness (see \cite[Corollary 8.16]{Klenke.2008} for an even more detailed version).

\begin{prop} \label{prop:optSol}
Let $(\Omega, \FF, P)$ be a complete probability space, $\GG$ be sub-$\sigma$-algebra of $\FF$. Let further $X, Z$ both be real-valued and square integrable random variables on $\Omega$ where in addition $Z$ is $\GG$-measurable. Then, the conditional expectation $\hat{Z}:=\E\left(X|\GG\right)$ is the minimizer of the mean-square distance from $X$
\begin{equation}
msd(X, Z) := \E\left(\left( X - Z\right)^2\right)
\end{equation}
among all such random variables $Z$.
\end{prop}
\begin{proof}
As $X$ is square integrable, the conditional expectation $\E\left(X|\GG\right)$ also is. Under the assumptions on $X$ and $Z$, we can calculate the mean-square distance explicitly as 
\begin{eqnarray*}
&&\E\left(\left( X - Z\right)^2\right) = \E\left(\left(X-\E\left(X|\GG\right)+ \E\left(X|\GG\right)- Z\right)^2\right) \\
&&= \E\left(\left(X- \E\left(X|\GG\right)\right)^2\right) + \E\left(\left( \E\left(X|\GG\right)- Z\right)^2\right) \\
&& \quad \quad + 2\E\left(\left(X - \E\left(X|\GG\right)\right)\left(\E\left(X|\GG\right)-Z\right)\right) \ .
\end{eqnarray*}
Using the fact that the expectation of the conditional expectation is the unconditional expectation and the $\GG$-measurability of $Z$ and $\E(X|\GG)$, we can condition on $\GG$ to obtain
\begin{eqnarray*}
&& 2\E\left(\left(X - \E\left(X|\GG\right)\right)\left(\E\left(X|\GG\right)-Z\right)\right) \\
&& =2\E\left[\E\left(\left(X - \E\left(X|\GG\right)\right)\left(\E\left(X|\GG\right)-Z\right)\right)|\GG\right]\\
&& = 2\E\left[\left(\left(\E\left(X|\GG\right)-Z\right)\E(X - X|\GG\right)|\GG\right] = 0 \ .
\end{eqnarray*}
Thus, the third term in the mean-square distance between $X$ and $Z$ always equals zero, the second term is always non-negative (as an expectation of a square), but zero for $Z = \hat{Z}= \E(X|\GG)$. As further the first term in the mean-square distance representation above is independent of $Z$, we have shown that $\hat{Z}$ is the minimizer of this distance.
\end{proof}

Note that we can calculate the optimal $L^2$-approximation of $X$ without the need to actually calculate the second moment of $X$ or the mean-square distance of $X$ from its conditional expectation.
We use the above general result in Subsection \ref{subsec:OCdiffset} for proving optimality of suitable conditional expectations of the demand process as the control strategy in different settings.

The main difference of the following approaches is the availability of demand information: The first control method (CM1) is based on a setting without demand updates at all, the second (CM2) is based on regular demand updates and the third (CM3) assumes an idealized demand knowledge. Those settings translate into different measurability assumptions for the control process $u(t)$:
\begin{enumerate}
\item[\textbf{CM1}]
\textbf{Setting without demand updates}

Most likely, the electricity producer has to decide in advance how much power is injected into the system at a later point in time. Other reasons for an early decision might be also high costs of obtaining demand information immediately or physical restrictions to directly use this continuously updated information for regulating the production intensity.

In the context of a short-term production optimization, except for the realized demand at the beginning of the production, one might assume that no updates on the available demand information can be incorporated into the control process. Thus, the task is to react in a best possible way given only the demand at the initial time. 

From a mathematical point of view, $u(t)$ is hence assumed to be $\FF_0$-predictable, i.e., the power injection for the whole period needs to be determined at the very beginning (remember that we also assume $\mu(\cdot)$ to be known in advance). Note that the information structure is modeled by the filtration $\FF_t \defgl \sigma\left(Y_s; 0\leq s \leq t\right)$.

How this affects the setting is shown in Figure \ref{fig:noUpdates}.
The only instance at which market information is included into the system is at the very beginning and is represented by the red line. This means that only at the very beginning, the demand actually realized at the market is taken into account. The chosen transportation time is $\Delta t\defgl \invLa =6$, i.e., it takes 6 units of time for the control $u(t)$ (upper timeline) to pass one unit of space (the system is normed to a length of 1) to be then considered as the output $y(t)$ (lower timeline). This transportation is represented by the blue arrows. Due to this time delay, the production is stopped $\Delta t$ time units before the end $T$ of the optimization period.

\item[\textbf{CM2}]
\textbf{Setting with regular demand updates}

On the long-run, it appears as a realistic scenario that, with a certain frequency, updates with respect to the actual demand realized at the market are incorporated into the optimization process (see Figure \ref{fig:regUpdates}).
We now assume that information about the current demand can only be updated at prespecified time points $0=\hatt_0<\hatt_1<\cdots<\hatt_n\le T-\nicefrac{1}{\lambda}$, where $\hatt_i=i \cdot \Delta t_{\text{up}}$, $i \in \{0,1,\cdots,\nicefrac{T-\invLa}{\Delta t_{\text{up}}}\}$, and update frequency $\Delta t_{\text{up}} \in [0,T-\invLa]$.

In Figure \ref{fig:regUpdates}, the instances at which new information is included into the system are again represented by red lines. The update frequency is set to $\Delta t^{\text{up}}=5$. This means that every 5 units of time, the demand actually realized at the market is observed and the forecasted value is adapted optimally given this new information. If the demand is significantly higher than the average, it is more likely that, in the short-run, the upcoming demand is still higher than the average demand and vice versa.

This setting translates into the measurability requirement that $u(t)$ has to be  $\FF_{\hatt_i}$-predictable for all $t \in [\hatt_i,\hatt_{i+1})$. Thus, the measurability assumption applies piecewise over time. 

We remark that the new information has no immediate impact on the system but acts with a time delay equal to the transportation time $\Delta t=\nicefrac{1}{\lambda}$.
As CM2 uses more information than CM1, the expected quadratic deviation of the output $y(t)$ from the actual demand $Y_t$ at time $t$ is smaller on average.

\item[\textbf{CM3}]
\textbf{Idealized setting}

In this setting, all demand information available at the market can be incorporated into an injection strategy instantaneously at any time.

This setting is visualized in Figure \ref{fig:idealSet}.
The current demand information with respect to the time scale of the injection (upper timeline) is available instantaneously at any time (light red background representing demand information from the market). This means that the demand actually realized at the market is observed at any time and results in an updated value for the optimal output based on the current market situation.
Hence, despite the continuous information basis, there is a time delay in reaction equal to the transportation time $\Delta t$.

As stated above, the available information for the controller in the idealized setting is the realized demand up to the current time $t$ given by $Y_s, 0\leq s \leq t$, and the demand dynamics.
In the idealized setting, the stochastic control process $u(t)$ is $\FF_t$-predictable.
As CM3 uses the most information compared to CM1 and CM2, the expected quadratic deviation of the output $y(t)$ from the actual demand $Y_t$ at time $t$ is the smallest on average.
\end{enumerate}

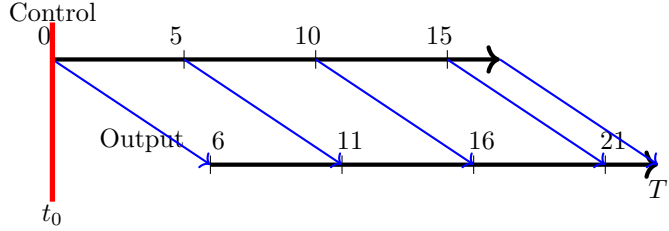
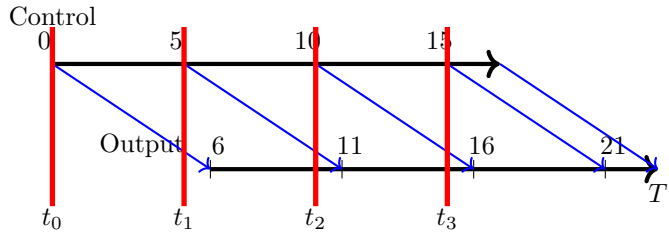
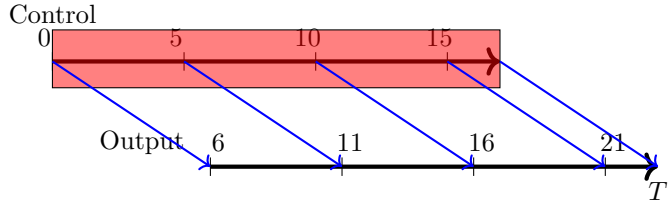
\begin{figure}[h!]
\centering
\subfloat[CM1: Setting without updates\label{fig:noUpdates}]{	\centering
	\tikzstyle{int}=[draw, fill=blue!20, minimum size=2em]
	\tikzstyle{init} = [pin edge={to-,thin,black}]
	
	\begin{tikzpicture}[scale=0.7]
	\foreach \x in {0,5,...,15}
	{        
		\coordinate (A\x) at ($(0,1)+(\x*0.5cm,1)$) {};
		\draw ($(A\x)+(0,5pt)$) -- ($(A\x)-(0,5pt)$);
		\node at ($(A\x)+(0,3ex)-(1ex,0)$) {\x};
	}
	
	\foreach \x in {6,11,16,17,21}
	{        
		\coordinate (A\x) at ($(0,1)+(\x*0.5cm,1)$) {};
	}
	\draw[ultra thick,arrows=->] (A0) -- (A17);
	
	\foreach \x in {6,11,16,21}
	{        
		\coordinate (B\x) at ($(0,0)+(\x*0.5cm,0)$) {};
		\draw ($(B\x)+(0,5pt)$) -- ($(B\x)-(0,5pt)$);
		\node at ($(B\x)+(1ex,3ex)$) {\x};
	}
	\coordinate (B23) at ($(0,0)+(23*0.5cm,0)$) {};
	\draw[ultra thick,arrows=->] (B6) -- (B23);
	\node at ($(B23)-(0,3ex)$) {$T$};
	
	\coordinate (B0) at ($(0,0)+(0*0.5cm,0)$) {};
	\foreach \x in {5,10,15,20}
	{        
		\coordinate (B\x) at ($(0,0)+(\x*0.5cm,0)$) {};
	}
	
	\node at ($(A0)+(0,6ex)$) {Control};
	\node at ($(B5)+(-0.8cm,3ex)$) {Output};
	
    \node at ($(B0)-(0,27pt)$) {$t_{0}$};
	
	\draw[blue,thick, arrows=->] (A0) -- (B6);
	\draw[blue,thick, arrows=->] (A5) -- (B11);
	\draw[blue,thick, arrows=->] (A10) -- (B16);
	\draw[blue,thick, arrows=->] (A15) -- (B21);
	\draw[blue,thick, arrows=->] (A17) -- (B23);
	
	\foreach \x in {0}
	{        
		\draw[red,line width=2pt] ($(A\x)+(0,20pt)$) -- ($(B\x)-(0,20pt)$);
	}
	\end{tikzpicture}}

	\subfloat[CM2: Setting with regular market updates\label{fig:regUpdates}]{
    \centering
	\tikzstyle{int}=[draw, fill=blue!20, minimum size=2em]
	\tikzstyle{init} = [pin edge={to-,thin,black}]
	
	\begin{tikzpicture}[scale=0.7]
	\foreach \x in {0,5,...,15}
	{        
		\coordinate (A\x) at ($(0,1)+(\x*0.5cm,1)$) {};
		\draw ($(A\x)+(0,5pt)$) -- ($(A\x)-(0,5pt)$);
		\node at ($(A\x)+(0,3ex)-(1ex,0)$) {\x};
	}
	
	\foreach \x in {6,11,16,17,21}
	{        
		\coordinate (A\x) at ($(0,1)+(\x*0.5cm,1)$) {};
	}
	\draw[ultra thick,arrows=->] (A0) -- (A17);
	
	\foreach \x in {6,11,16,21}
	{        
		\coordinate (B\x) at ($(0,0)+(\x*0.5cm,0)$) {};
		\draw ($(B\x)+(0,5pt)$) -- ($(B\x)-(0,5pt)$);
		\node at ($(B\x)+(1ex,3ex)$) {\x};
	}
	\coordinate (B23) at ($(0,0)+(23*0.5cm,0)$) {};
	\draw[ultra thick,arrows=->] (B6) -- (B23);
	\node at ($(B23)-(0,3ex)$) {$T$};
	
	\coordinate (B0) at ($(0,0)+(0*0.5cm,0)$) {};
	\foreach \x in {5,10,15,20}
	{        
		\coordinate (B\x) at ($(0,0)+(\x*0.5cm,0)$) {};
	}
	
	\node at ($(A0)+(0,6ex)$) {Control};
    \node at ($(B5)+(-0.8cm,3ex)$) {Output};
	
    \node at ($(B0)-(0,27pt)$) {$t_{0}$};
	\node at ($(B5)-(0,27pt)$) {$t_{1}$};
	\node at ($(B10)-(0,27pt)$) {$t_{2}$};
	\node at ($(B15)-(0,27pt)$) {$t_{3}$};
	
	\draw[blue,thick, arrows=->] (A0) -- (B6);
	\draw[blue,thick, arrows=->] (A5) -- (B11);
	\draw[blue,thick, arrows=->] (A10) -- (B16);
	\draw[blue,thick, arrows=->] (A15) -- (B21);
	\draw[blue,thick, arrows=->] (A17) -- (B23);
	
	\foreach \x in {0,5,10,15}
	{        
		\draw[red,line width=2pt] ($(A\x)+(0,20pt)$) -- ($(B\x)-(0,20pt)$);
	}
	
	\end{tikzpicture}}

\subfloat[CM3: Idealized setting \label{fig:idealSet}]{
	\centering
	\tikzstyle{int}=[draw, fill=blue!20, minimum size=2em]
	\tikzstyle{init} = [pin edge={to-,thin,black}]
	
	\begin{tikzpicture}[scale=0.7]
	\foreach \x in {0,5,...,15}
	{        
		\coordinate (A\x) at ($(0,1)+(\x*0.5cm,1)$) {};
		\draw ($(A\x)+(0,5pt)$) -- ($(A\x)-(0,5pt)$);
		\node at ($(A\x)+(0,3ex)-(1ex,0)$) {\x};
	}
	
	\foreach \x in {6,11,16,17,21}
	{        
		\coordinate (A\x) at ($(0,1)+(\x*0.5cm,1)$) {};
	}
	\draw[ultra thick,arrows=->] (A0) -- (A17);
	
	\foreach \x in {6,11,16,21}
	{        
		\coordinate (B\x) at ($(0,0)+(\x*0.5cm,0)$) {};
		\draw ($(B\x)+(0,5pt)$) -- ($(B\x)-(0,5pt)$);
		\node at ($(B\x)+(1ex,3ex)$) {\x};
	}
	\coordinate (B23) at ($(0,0)+(23*0.5cm,0)$) {};
	\draw[ultra thick,arrows=->] (B6) -- (B23);
	\node at ($(B23)-(0,3ex)$) {$T$};
	
    \draw[fill = red, fill opacity = 0.5] (0,1.5)  rectangle  (8.5,2.6);

	\coordinate (B0) at ($(0,0)+(0*0.5cm,0)$) {};
	\foreach \x in {5,10,15,20}
	{        
		\coordinate (B\x) at ($(0,0)+(\x*0.5cm,0)$) {};
	}
	
	\node at ($(A0)+(0,6ex)$) {Control};
	\node at ($(B5)+(-0.8cm,3ex)$) {Output};

	\draw[blue,thick, arrows=->] (A0) -- (B6);
	\draw[blue,thick, arrows=->] (A5) -- (B11);
	\draw[blue,thick, arrows=->] (A10) -- (B16);
	\draw[blue,thick, arrows=->] (A15) -- (B21);
	\draw[blue,thick, arrows=->] (A17) -- (B23);
	\end{tikzpicture}}
	\caption{Update algorithms CM1--CM3 with transportation time $\invLa =6$}
	\label{fig:funcUpdateAlg}
\end{figure}

\subsection{Optimal control for different settings} \label{subsec:OCdiffset}
Next, we derive explicit expressions for the optimal control under the different measu\-rability assumptions corresponding to the different settings CM1--CM3. The results are summarized in the following theorem.
\begin{thm}\label{thm:controls}
Let us assume the demand is a jump diffusion process \eqref{eq:nonCompJDP}.
Then, given a time-homogeneous jump height process $\left(\gamma_t\right)_{t \in [0,T]}$, where $\gamma_{t}$ is distributed according to a given distribution with existing second moment, and $\bar{\gamma}=\E(\gamma_t)$, the optimal control is given
    \begin{enumerate}
        \item for u(t) being $\FF_0$-measurable in the setting without updates (CM1) as
        \begin{align}
		u^\ast(t;t_0) =& e^{-\kappa (t+\invLa)}y_0 + \kappa 		\int_{0}^{t+\invLa} \exp\left(-\kappa (t+\invLa-s)\right) \mu(s) ds + \frac{\bar{\gamma} \nu}{\kappa} \left(1-e^{-\kappa (t+\invLa)}\right), \label{eq:firstMomJDPExp}
		\end{align}
        \item for u(t) being $\FF_{\hatt_i}$-measurable in the setting with regular market updates (CM2) as
        \begin{align}
		& u^{\ast}(t;\hatt_i) =  e^{-\kappa(t+\invLa-\hatt_i)} Y_{\hatt_i} + \kappa\int_{\hatt_i}^{t+\invLa}e^{-\kappa(t+\invLa-s)} \mu(s) ds + \frac{\bar{\gamma} \nu}{\kappa} \left(1-e^{-\kappa (t+\invLa - \hatt_i)}\right), \label{eq:optControlJDPregUpdates}
		\end{align}
        \item and for u(t) being $\FF_t$-measurable in the idealized setting (CM3) as
        \begin{align}
		u^{\ast}(t) =& e^{-\nicefrac{\kappa}{\lambda}} Y_{t} + \kappa\int_{t}^{t+			\invLa}e^{-\kappa(t+\invLa-s)} \mu(s) ds + \frac{\bar{\gamma} \nu}{\kappa}\left(1-e^{-\nicefrac{\kappa}{\lambda}}\right). \label{eq:optControlJDP}
		\end{align}
    \end{enumerate}
\end{thm}

\begin{proof}
The key idea of the proof is to choose the $\sigma$-algebra $\GG$ in Proposition \ref{prop:optSol} that corresponds to the underlying setting and then calculate the resulting conditional expectation explicitly.
It is advantageous, first to prove the claim for control method CM3 and then discuss the solutions to CM1 and CM2.
\begin{itemize}
\item[\em{3.}] 
Using the choice of  
$\GG = \FF_{t-\nicefrac{1}{\lambda}}$
in Proposition \ref{prop:optSol}, we directly obtain that the optimal $y^\ast(t)$ for CM3 is given by 
\[
y^\ast(t) = \E\left(Y_t|\FF_{t-\nicefrac{1}{\lambda}}\right)\ .
\]
The expected cost over the time span $[0, T]$ for each point in time $t$ can be minimized.
Due to the explicit solution to \eqref{eq:transportDyn}
we know that the corresponding control 
$u^\ast(t) = y^\ast(t+\invLa)$
is globally feasible as it is $\FF_t$-predictable and square-integrable.
So, the remaining task is now to compute the conditional expectation
$u^\ast(t)= \E\left[Y_{t+\invLa} | \FF_{t}\right]\ .$

The optimal solution of \eqref{eq:SOC-Prob} - \eqref{eq:JDPinOptProb} results from the following calculation, where we plug in the explicit solution \eqref{eq:solJDP} of the JDP.
\begin{align}
u^{\ast}(t) =& \E\left[Y_{t+\nicefrac{1}{\lambda}} | \FF_{t}\right] \notag\\
=&e^{-\nicefrac{\kappa}{\lambda}} Y_{t} + \E\left[\kappa\int_{t}^{t+\invLa}e^{-\kappa(t+\invLa-s)} \mu(s) ds \right.  \notag\\
&\left.+ \sigma \int_{t}^{t+\invLa} e^{-\kappa(t+\invLa-u)}dW_u + \sum_{i=N_{t}+1}^{N_{t+\invLa}} \gamma_{t_i} e^{-\kappa (t+\invLa-t_i)} | \FF_{t}\right] \notag\\
=&e^{-\nicefrac{\kappa}{\lambda}} Y_{t} + \kappa\int_{t}^{t+\invLa}e^{-\kappa(t+\invLa-s)} \mu(s) ds + \bar{\gamma}\E\left[\sum_{i=N_{t}+1}^{N_{t+\invLa}} e^{-\kappa (t+\invLa-t_i)} \right] \label{eq:condExpJump}
\end{align}
Note that we have used the martingale property of the stochastic integral as the deterministic integrand is square-integrable. Furthermore, we have exploited the Markov property of the JDP and the independence of the jump height from the other movements.

It remains to calculate the last expectation in \eqref{eq:condExpJump}. First, we condition on $N_{t+\nicefrac{1}{\lambda}}$ and $N_t$, respectively. We know that, conditional on the number of jumps of a Poisson process in the interval $[t, t+\nicefrac{1}{\lambda}]$, the jump times $t_i$ are all independently and identically $\UU([t,t+1/\lambda])$ distributed random variables (\cite[Prop.\ 2.1.16]{Mikosch.2009}). Thus, we have $\left(t+\nicefrac{1}{\lambda} - t_i\right)$ $\sim\UU(0,1/\lambda)$ and get: 
\begin{align}
& \E\left[\sum_{i=N_{t}+1}^{N_{t+\invLa}} {e^{-\kappa (t+\invLa-t_i)}}\right] = \sum_{i=1}^{\infty}\left(i \int_0^{1/\lambda}{e^{-\kappa x}\lambda dx}\;\cdot e^{-\nicefrac{\nu}{\lambda}} \frac{\left(\nicefrac{\nu}{\lambda}\right)^i}{i!}\right) = \frac{\nu}{\kappa}\left(1-e^{-\nicefrac{\kappa}{\lambda}}\right) \label{eq:firstMomExpJumpPart}
\end{align}
Plugging \eqref{eq:firstMomExpJumpPart} into \eqref{eq:condExpJump} gives the optimal control in case of a JDP-type demand for CM3:
\begin{align*}
u^{\ast}(t) =& e^{-\nicefrac{\kappa}{\lambda}} Y_{t} + \kappa\int_{t}^{t+\invLa}e^{-\kappa(t+\invLa-s)} \mu(s) ds + \frac{\bar{\gamma} \nu}{\kappa}\left(1-e^{-\nicefrac{\kappa}{\lambda}}\right).
\end{align*}
Thus, the optimal control can directly be calculated from the current demand and the weighted forecasted demand over the next $\invLa$ time units.
\item[\em{1.}] 
	For control method CM1, we use Proposition \ref{prop:optSol} with the choice of 
	$\GG = \FF_{0}$
	independently of the time $t \in [0,T]$, and we directly obtain that the optimal value of $y^\ast(t)$ is given by 
\[
y^\ast(t) = \E\left(Y_t|\FF_{0}\right) = \E\left(Y_t\right) \ .
\]
As in the idealized setting CM3, the corresponding optimal control is
$
u^\ast(t;t_0) = y^\ast(t+\invLa) \ .
$
To calculate the expected value of the
jump diffusion process \eqref{eq:nonCompJDP}, we proceed similarly to the calculations done before.
\begin{align}
\E\left[Y_t\right] =& e^{-\kappa t}y_0 + \sigma \E\left[\int_{0}^{t} e^{-\kappa (t-s)} dW_s\right] + \kappa \int_{0}^{t} e^{-\kappa (t-s)}\mu(s) ds + \bar{\gamma} \E\left[\sum_{i=1}^{N_t} e^{-\kappa (t-t_i)}\right] \notag\\
=& e^{-\kappa t}y_0 + \kappa \int_{0}^{t} e^{-\kappa (t-s)} \mu(s) ds + \bar{\gamma}\E\left[\sum_{i=1}^{N_t} e^{-\kappa (t-t_i)}\right].\label{eq:firstMomJDP}
\end{align}
The last expectation in \eqref{eq:firstMomJDP} can be obtained as before
\begin{align}
\E\left[\sum_{i=1}^{N_t} e^{-\kappa (t-t_i)}\right] &= \frac{\nu}{\kappa}\left(1-e^{-\kappa t}\right). \label{eq:expJumpPartFirstMomExp}
\end{align}

Finally, by plugging \eqref{eq:expJumpPartFirstMomExp} into \eqref{eq:firstMomJDP}, we obtain the optimal control for CM1, i.e. in the setting without updates, as
\begin{align*}
	u^\ast(t;t_0) =& e^{-\kappa\left( t+\invLa\right)}y_0 + \kappa \int_{0}^{t+\invLa} e^{-\kappa \left(t+\invLa-s\right)}\mu(s) ds + \frac{\bar{\gamma}\nu}{\kappa}\left(1-e^{-\kappa\left( t+\invLa\right)}\right).
\end{align*}

\item[\em{2.}]
We consider now that the information is only given by $\FF_{\hatt_i}$ and the deterministic demand forecast $\mu(s)$, $s\in [0,t]$.
By then calculating 
\[
u^\ast(t;\hatt_i) = y^\ast(t+\invLa;\hatt_i) = \E\left(Y_{t+\invLa}|\FF_{\hatt_i}\right), \quad \forall \ t \in [\hatt_i,\hatt_{i+1}),
\]
as in the idealized setting CM3, we obtain the claimed form of the optimal control also for CM2. \qedhere
\end{itemize} 
\end{proof}
 \begin{rem}
 The corresponding controls applying the OUP \eqref{eq:OUP} for the demand process are obtained by setting $\bar{\gamma} \equiv 0$ in Theorem \ref{thm:controls}.
 \end{rem}

\subsection{Relation between the different approaches} \label{subsec:RelationApproaches}

The three control approaches CM1--CM3 are related to each other via the frequency of the updates on the actual power demand. The approach without updates (CM1) can be derived by setting the time between two updates as $\Delta t_{\text{up}} > T-\invLa$ in CM2. The latter involves that the only relevant update time is $t_0=0$, i.\ e., we require that $u(t)$ is $\FF_0$-measurable for all $t \in [0,T-\invLa]$. 

In the following theorem, we show a kind of consistency condition for the time between two updates $\Delta t_{\text{up}}$ getting infinitesimally small. 

\begin{thm}\label{thm:convergenceUpToTheo}
The optimal control in the idealized setting CM3 is the limit of the optimal control with updates at the discrete times $\hatt_i=i \Delta t_{\text{up}}$ if the time between the updates $\Delta t_{\text{up}}$ tends to zero, i.e.
$$\lim_{\Delta t_{\text{up}} \rightarrow 0} u^{\ast}(t)-u^{\ast}(t;\hatt_i) = 0 \ \ \Prob-a.s.\;.$$
\end{thm}

\begin{proof}
By plugging \eqref{eq:solJDP} into \eqref{eq:optControlJDPregUpdates} in the case of regular updates (CM2) leads to 
\begin{align}
u^{\ast}(t;\hatt_i) =& y_0 e^{-\kappa (t+\invLa)} + \kappa \int_{0}^{t+\invLa} \mu(s)e^{-\kappa (t+\invLa-s)} ds \notag\\
&+ \sigma \int_{0}^{\hatt_i} e^{-\kappa (t+\invLa-s)} dW_s + \gamma \sum_{j=1}^{N_{\hatt_i}} e^{-\kappa (t+\invLa-t_j)} + \frac{\bar{\gamma} \nu}{\kappa} \left(1-e^{-\kappa (t+\invLa - \hatt_i)}\right), \label{eq:OCJDPregUpExplicit}
\end{align}
and by plugging \eqref{eq:solJDP} into \eqref{eq:optControlJDP}, the optimal control in the idealized setting CM3 reads
\begin{align}
u^{\ast}(t) =& y_0e^{-\kappa (t+\invLa)} + \kappa \int_{0}^{t+\invLa} \mu(s)e^{-\kappa (t+\invLa-s)} ds \notag\\
&+ \sigma \int_{0}^{t} e^{-\kappa (t+\invLa-s)} dW_s + \gamma \sum_{j=1}^{N_t} e^{-\kappa (t+\invLa-t_j)} + \frac{\bar{\gamma} \nu}{\kappa}\left(1-e^{-\nicefrac{\kappa}{\lambda}}\right). \label{eq:OCJDPExplicit}
\end{align}
Then, the difference is
\begin{align}
u^{\ast}(t)-u^{\ast}(t;\hatt_i) =&  \sigma \left(\int_{0}^{t} e^{-\kappa (t+\invLa-s)} dW_s - \int_{0}^{\hatt_i} e^{-\kappa (t+\invLa-s)} dW_s\right) \notag\\
&+ \gamma \left( \sum_{j=1}^{N_t} e^{-\kappa (t+\invLa-t_j)}-\sum_{j=1}^{N_{\hatt_i}} e^{-\kappa (t+\invLa-t_j)}\right) + \frac{\bar{\gamma} \nu}{\kappa} \left(e^{-\kappa (t+\invLa - \hatt_i)} - e^{-\nicefrac{\kappa}{\lambda}}\right). \label{eq:diffOCJDP}
\end{align}

Note that both It\^o-integrals are pathwise continuous stochastic processes. 
Further, the Poisson process is pathwise finite and admits c\`{a}dl\`{a}g-paths, i.e.\ right-continuous paths with existing left-sided limit, almost surely. 
Thus,  for almost every $\omega \in \Omega$, the left and right sides of \eqref{eq:diffOCJDP} are c\`{a}dl\`{a}g functions.

For each $\Delta t_{\text{up}}$ and an arbitrary but fixed $\ttt$, we can find $i \in \{0,1,\cdots,\nicefrac{T-\invLa}{\Delta t_{\text{up}}}\}$ such that $\ttt \in [\hatt_i,\hatt_{i+1})$, and we have
\begin{align*}
0 \leq \ttt-\hatt_i \leq \Delta t_{\text{up}}.
\end{align*}
Hence, as $\Delta t_{\text{up}}$ tends to zero, $\hatt_i$ tends to $\ttt$, and we have
\begin{align*}
\lim_{\Delta t_{\text{up}} \rightarrow 0} \lvert u^{\ast}(\ttt)-u^{\ast}(\ttt;\hatt_i)\rvert = \lim_{\hatt_i \rightarrow \ttt} \lvert u^{\ast}(\ttt)-u^{\ast}(\ttt;\hatt_i)\rvert.
\end{align*}
Due to the cadl\`{a}g-property of \eqref{eq:diffOCJDP} and the equality of the above limits, we obtain the convergence of the optimal control in the update setting towards the optimal control in the idealized setting.
\end{proof}

\section{Numerical results for the SOC} \label{sec:Numerics}
In this section, we analyze the optimal control problem \eqref{eq:SOC-Prob}-\eqref{eq:JDPinOptProb} from a numerical point of view. To make the problem numerically tractable with classical optimization algorithms, we first derive an explicit expression of the cost function \eqref{eq:OF}. A detailed numerical study allows then to investigate the strategies CM1--CM3 for the different types of stochastic demands.   

\subsection{Analytical treatment of the cost function}\label{subsec:analyticalTreatmentCost}

As it turns out, the cost function \eqref{eq:OF} can be expressed in a direct way.
Using the reformulated cost function combined with constraint \eqref{eq:transportDyn}, the original SOC \eqref{eq:SOC-Prob} - \eqref{eq:JDPinOptProb} reduces to a deterministic optimal control problem of demand tracking. Similar problems have already been investigated in \cite{Goettlich2018,Lamarca2010}.

We start by replacing the expectation in the objective function \eqref{eq:OF} by an explicit expression in terms of the first two moments of the demand process:
\begin{align}
	\E\left[(Y_t - y(t))^2 \right] &= \E\left[Y_t^2\right] - 2y(t)\E\left[Y_t\right] +y(t)^2. \label{eq:costFctMomRep}
\end{align}
This approach allows to incorporate the demand dynamics directly into the objective function as those two quantities are the only characteristics of the demand process influencing the objective function. As a consequence, for this particular cost structure, the stochastic demand constraint \eqref{eq:JDPinOptProb} is no longer needed and the control problem is only restricted by the transport equation \eqref{eq:transportDyn}.

The first moment in case of an OUP-type or a JDP-type demand have already been calculated in the setting without updates, see Subsection \ref{subsec:OCdiffset}.
Thus, in order to get an explicit representation of the cost function, it remains to calculate the second moment of the demand process. Details of the calculation can be found in the Appendix \ref{app:secondMomJDP}.
\begin{align}
\E\left[Y_t^2\right] =& \E\left[\left(e^{-\kappa t}y_0 + \sigma \int_{0}^{t} e^{-\kappa (t-s)} dW_s + \kappa \int_{0}^{t} e^{-\kappa (t-s)}\mu(s) ds + \gamma \sum_{i=1}^{N_t} e^{-\kappa (t-t_i)}\right)^2\right] \notag\\
=& \left(e^{-\kappa t}y_0 + \kappa \int_{0}^{t} e^{-\kappa (t-s)}\mu(s) ds\right)^2 + \frac{\sigma^2}{2\kappa}\left(1-e^{-2 \kappa t}\right) + \E\left[\left(\sum_{i=1}^{N_t}\gamma_{t_i} e^{-\kappa (t-t_i)}\right)^2\right] \notag\\
&+ 2\left(e^{-\kappa t}y_0 + \kappa \int_{0}^{t} e^{-\kappa (t-s)}\mu(s) ds\right) \cdot \bar{\gamma} \frac{\nu}{\kappa}\left(1-e^{-\kappa t}\right). \label{eq:secMomJDP}
\end{align}

Now, it remains to compute $\E\left[\left(\sum_{i=1}^{N_t}\gamma_{t_i} e^{-\kappa (t-t_i)}\right)^2\right] $.
\begin{align*}
	\E\left[\left(\sum_{i=1}^{N_t}\gamma_{t_i} e^{-\kappa (t-t_i)}\right)^2\right] &= \nu \cdot \frac{(1-e^{-2\kappa t})}{2\kappa} \cdot \E\left[\gamma^2\right] + \nu^2 \cdot \frac{1+e^{-2\kappa t}-2e^{-\kappa t}}{\kappa^2} \cdot \bar{\gamma}^2.
\end{align*}

The second moment of the JDP \eqref{eq:nonCompJDP} consequently reads as
\begin{align}
	\E\left[Y_t^2\right] &= y_0^2e^{-2\kappa t} + 2y_0e^{-\kappa t} \int_{0}^{t}e^{-\kappa(t-s)}\kappa \mu(s) ds + \left(\kappa\int_{0}^{t}e^{-\kappa(t-s)} \mu(s) ds\right)^2 \notag\\
	&+ \frac{\sigma^2}{2\kappa}\left(1-e^{-2 \kappa t}\right) + \nu \cdot \frac{(1-e^{-2\kappa t})}{2\kappa} \cdot \E\left[\gamma^2\right] + \nu^2 \cdot \frac{1+e^{-2\kappa t}-2e^{-\kappa t}}{\kappa^2} \cdot \bar{\gamma}^2 \notag\\
	&+ 2\left(e^{-\kappa t}y_0 + \kappa \int_{0}^{t} e^{-\kappa (t-s)} \mu(s) ds\right) \cdot \bar{\gamma} \frac{\nu}{\kappa}\left(1-e^{-\kappa t}\right). \label{eq:secMomJDPExp}
\end{align}

Summarizing, the complete deterministic reformulation of the cost term \eqref{eq:costFctMomRep} based on the first two moments (cf. \eqref{eq:firstMomJDPExp} and \eqref{eq:secMomJDPExp}) of the JDP is given by 
\begin{align}
	\E\left[(Y_t - y(t))^2 \right] =& y_0^2e^{-2\kappa t} + 2y_0e^{-\kappa t} \int_{0}^{t}e^{-\kappa(t-s)}\kappa \mu(s) ds + \left(\kappa\int_{0}^{t}e^{-\kappa(t-s)} \mu(s) ds\right)^2 \notag\\
	&+ \frac{\sigma^2}{2\kappa}\left(1-e^{-2 \kappa t}\right) + + \nu \cdot \frac{(1-e^{-2\kappa t})}{2\kappa} \cdot \E\left[\gamma^2\right] + \nu^2 \cdot \frac{1+e^{-2\kappa t}-2e^{-\kappa t}}{\kappa^2} \cdot \bar{\gamma}^2 \notag\\
	&+ 2\left(e^{-\kappa t}y_0 + \kappa \int_{0}^{t} e^{-\kappa (t-s)} \mu(s) ds\right) \cdot \bar{\gamma}\frac{\nu}{\kappa}\left(1-e^{-\kappa t}\right) \notag\\
	&- 2y(t) \cdot \left( e^{-\kappa t}y_0 + \kappa \int_{0}^{t} e^{-\kappa (t-s)}\mu(s) ds \right. \notag\\
	&\left.+ \bar{\gamma} \frac{\nu}{\kappa}\left(1-e^{-\kappa t}\right)\right) + y(t)^2. \label{eq:costFctMomRepExpJDP}
\end{align}

As already mentioned, the explicit expression of the cost function for an OUP-type demand can be obtained by setting the jump height $\gamma_t \equiv 0$ in \eqref{eq:costFctMomRepExpJDP}.

\subsection{Numerical investigation of the optimal control} \label{subsec:numInvestOC}
By plugging in \eqref{eq:costFctMomRepExpJDP} in \eqref{eq:OF}, we are left with the optimal control problem \eqref{eq:SOC-Prob}-\eqref{eq:control}.
This means, we are able to apply common optimization tools to solve the problem numerically.
In our case, we use the nonlinear optimization solver \textit{fmincon} from MATLAB R2015b\footnote{\url{https://de.mathworks.com/help/optim/ug/fmincon.html}, last checked: Sept 21, 2018}.
The key idea of the numerical experiments is to see how the numerical solution performs compared to the closed solution expressions in Theorem \ref{thm:controls}, in particular in the case of market updates. 

The computational setting is as follows: The transport equation \eqref{eq:transportDyn} is discretized straightforward using a left-sided Upwind scheme \cite{LeVeque1990}, i.e. 
$$
\frac{z(x_j,\tau_{i+1})-z(x_j,\tau_i)}{\Delta \tau} + \lambda \frac{z(x_j,\tau_i)-z(x_{j-1},\tau_i)}{\Delta x}=0.
$$
The step sizes $\Delta x=0.1$, and $\Delta \tau=\nicefrac{\Delta x}{\lambda}$ are chosen in a way such that stability (i.e. the CFL condition) for the transport equation is satisfied.
Furthermore, we assume an empty system at the beginning, i.\ e. $z_0(x_j) = 0$ for all $x_j \in (0,1)$. 
\subsubsection{Deterministic demand}
For validation purposes, we first perform computations for $Y_t = 2 + \sin(0.5\pi t)$ for $0.5\leq t\leq 5$, i.e.\ a purely deterministic demand, with $\Delta x=0.5$. Thereby, the transport velocity is chosen to be $\lambda=2$ and the total time horizon is $T=5$. 
In Figure \ref{fig:detdemand2plussin05pitlambda2t5feval6000}, we observe that the optimal control is close to the deterministic demand shifted by the transportation time $\nicefrac{1}{\lambda}=0.5$ and $u(\tau_i)=y(\tau_i+\nicefrac{1}{\lambda})$ holds. 
\begin{figure}[h!]
\centering
\includegraphics[width=0.5\textwidth]{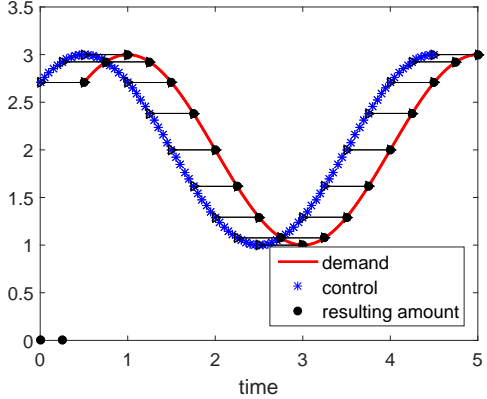}
\caption{Optimal control and available power in a deterministic demand setting}
\label{fig:detdemand2plussin05pitlambda2t5feval6000}
\end{figure}

\subsubsection{Stochastic demand}
We continue our numerical study and, in a next step, address the setting with stochastic demands supplemented with the strategies CM1 (without market updates) and CM2 (with regular market updates). 
In order to use the \textit{fmincon} solver together with the deterministic reformulation of the cost function including updates, we need to slightly modify the control problem for the update setting. 
We consider a partition of the interval $[0,T]$ in subintervals $[\hatt_i,\hatt_{i+1}]$, where $\hatt_i=i \cdot \Delta t_{\text{up}}$, \newline $i \in \{0,1,\cdots,\nicefrac{T-\invLa}{\Delta t^{\text{up}}}\}$ and update frequency $\Delta t^{\text{up}} \in [0,T]$.
Then, we solve the following sequence of optimization problems determined by the subintervals with \textit{fmincon} again.
\begin{align}
	\min_{u(t), t \in [\hatt_i,\hatt_{i+1}]} & \int_{\hatt_i + \invLa}^{\min\{\hatt_{i+1}+\invLa,T\}} \E\left[(Y_t - y(t))^2 | \FF_{\hatt_i}\right] dt \notag \\
	z_t + \lambda z_x & = 0, \quad z(0,t) = u(t), \notag \\
	z(x,\hatt_i) & =  z_{\text{old}}(x,\hatt_i), \quad x \in (0,1), t \in [\hatt_i,\min	\{\hatt_{i+1}+\invLa,T\}], \label{eq:SOC-Prob_update}
\end{align}
where $z_{\text{old}}(x,\hatt_i)$ is the state of the system at update time $\hatt_i$. Note that $z_{\text{old}}(x,\hatt_i)$ is equal to $z_0(x)$ for $i=0$ but needs to be assigned the possibly different state of the system at each update time $\hatt_i$.

The remainder is now three-fold: First, we show numerical results for an OUP-type demand in both settings. Second, we present corresponding results for a JDP-type demand. Third, we conclude the numerical study with a validation of the numerical solution against the analytical one from Section \ref{sec:TheoreticalOCapproaches}.

The parameter settings are given in Table \ref{tab:ParaPS}. The settings only differ in the speed of mean reversion $\kappa$, and the constantly chosen jump height $\gamma_t \equiv \gamma$.
\begin{table}[h!]
	\centering
	\begin{tabular}{lcrrr}
		\hline 
		Parameter &  & $PS1$ & $PS2$ & $PS3$ \\ 
		\hline 
		Transport velocity & $\lambda$ & 4 & $PS1$ & $PS1$ \\ 
		Time horizon & $T$ & 1 & $PS1$ & $PS1$ \\ 
		Mean demand level & $\mu(t)$ & $2 + 3 \cdot \sin(2\pi t)$ & $PS1$ & $PS1$ \\
		Speed of mean reversion & $\kappa$ & \textbf{1} & \textbf{3} & \textbf{3} \\
		Intensity of demand fluctuations & $\sigma$ & 2 & $PS1$ & $PS1$  \\ 
		Initial demand & $y_0$ & 1 & $PS1$ & $PS1$  \\
        Jump height & $\gamma$ & \textbf{0} & $PS1$ & \textbf{1}\\
        Jump intensity & $\nu$ & 5 & $PS1$ & $PS1$ \\ 
		\hline 
	\end{tabular}
	\caption{Parameter settings}
	\label{tab:ParaPS}
\end{table}

We investigate the average match between the available power at $x=1$ and the realized demand based on the different information scenarios CM1 and CM2. In those figures, we depict the available power by a black dash-dotted line, the optimal control by a blue dotted line and the mean demand by a thicker dashed blue line.

To get further insight how the demand process, starting in $y_0$, stochastically evolves over time, we plot the confidence levels of the demand process in grey scale. Furthermore, we add the resulting available power from the optimal control (black dash-dotted), and the mean realization of the OUP-demand (blue-dashed). 

\subsubsection*{OUP-type demand: Parameter settings PS1 and PS2}
The optimal transport processes are depicted in Figures \ref{fig:transportNoUpdates_lambda4_initialFix1_2Plus3Malsin2pit_kappa1_tildeSigma2_MC10P3} and \ref{fig:transportNoUpdates_lambda4_initialFix1_2Plus3Malsin2pit_kappa3_tildeSigma2_MC10P3}. We can deduce from those figures that for increasing values of $\kappa$, the shape of the control and, thus, the available power, reflects stronger the sine-shaped mean demand level.
\begin{figure}[h!]
	\subfloat[\ $PS1$ \label{fig:transportNoUpdates_lambda4_initialFix1_2Plus3Malsin2pit_kappa1_tildeSigma2_MC10P3}]{\includegraphics[width=0.5\linewidth]{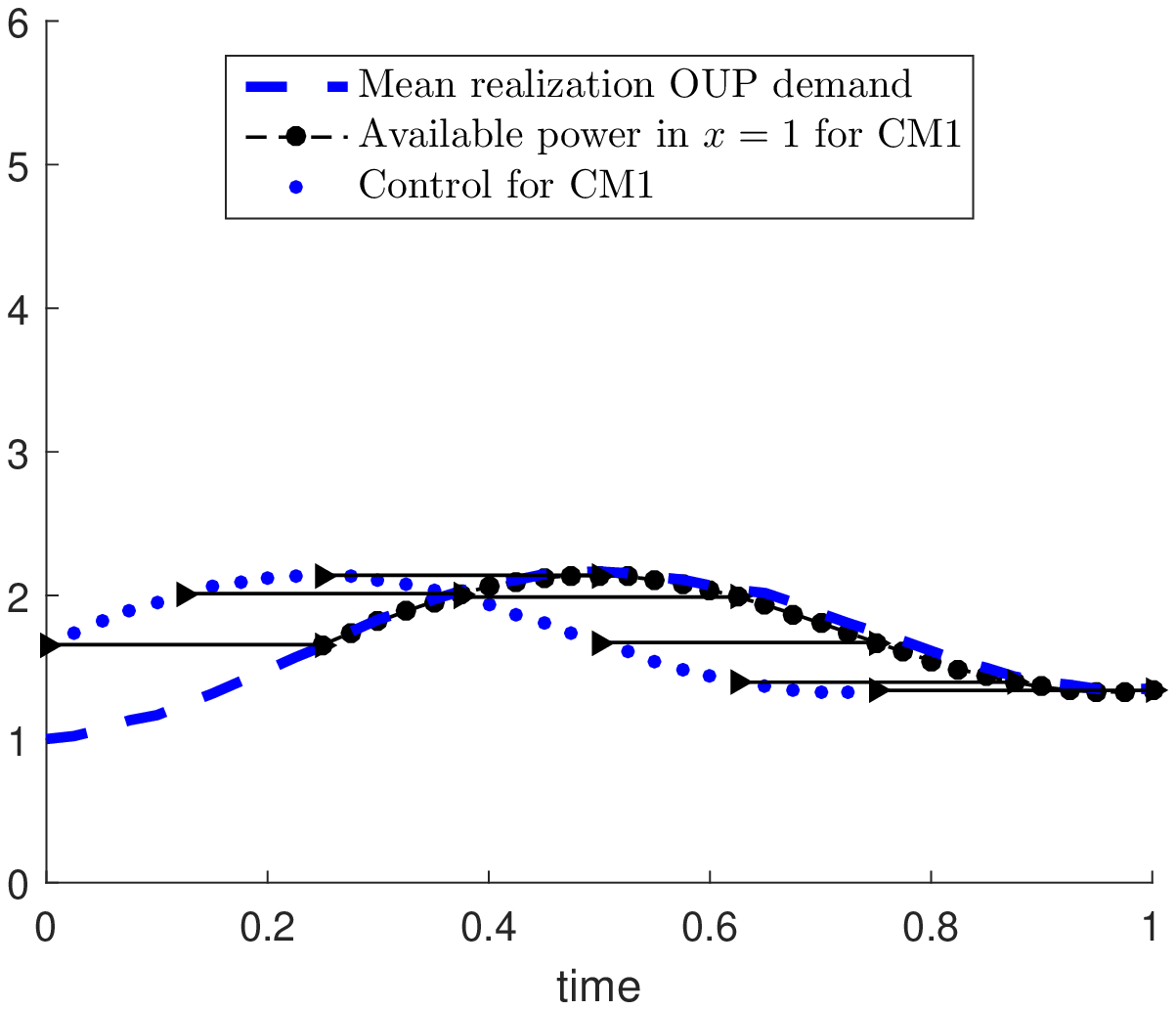}}\hfill
    \subfloat[\ $PS2$ \label{fig:transportNoUpdates_lambda4_initialFix1_2Plus3Malsin2pit_kappa3_tildeSigma2_MC10P3}]{\includegraphics[width=0.5\linewidth]{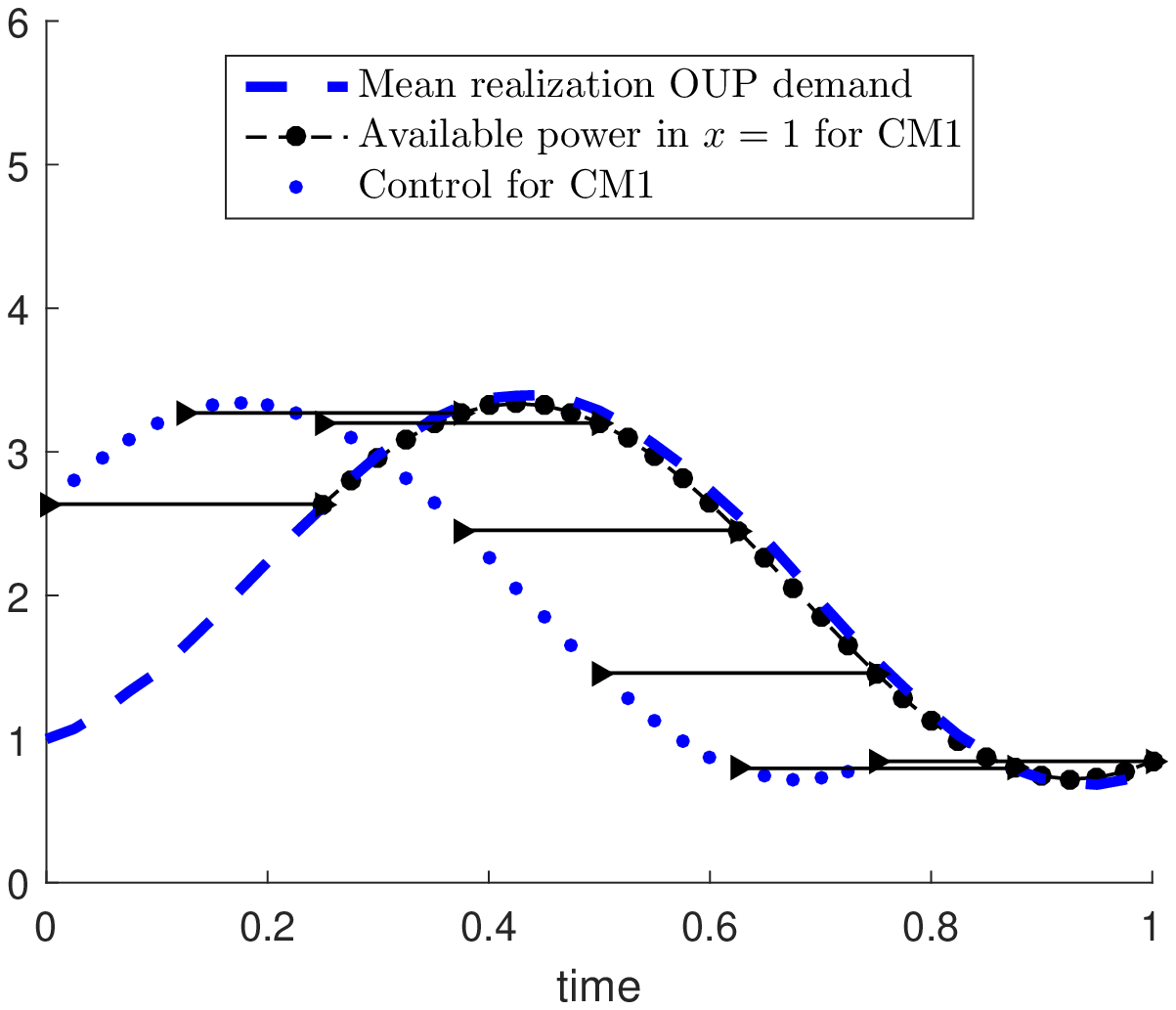}}
	\caption{Optimal control, output and mean realization of demand}
\end{figure}

In Figures \ref{fig:outputOUPnoUpdates_lambda4_initialFix1_2Plus3Malsin2pit_kappa1_tildeSigma2_MC10P3} and \ref{fig:outputOUPnoUpdates_lambda4_initialFix1_2Plus3Malsin2pit_kappa3_tildeSigma2_MC10P3}, 
we recognize that for higher speeds of mean reversion $\kappa$, the realizations of the OUP-type demand are less spread and more concentrated around the mean realization of the demand. Demand forecasts for a low value of $\kappa$ are difficult due to a flat mean realization of the demand and large fluctuations around the latter. Those fluctuations arise from the low attraction to the mean demand level. In this setting, we will see that there is an even greater benefit of control strategy CM2. This is due to the incorporation of actual demand information from market data into the forecast and, hence, also into the optimal control.
\begin{figure}[h!]
	\subfloat[\ $PS1$ \label{fig:outputOUPnoUpdates_lambda4_initialFix1_2Plus3Malsin2pit_kappa1_tildeSigma2_MC10P3}]{\includegraphics[width=0.5\linewidth]{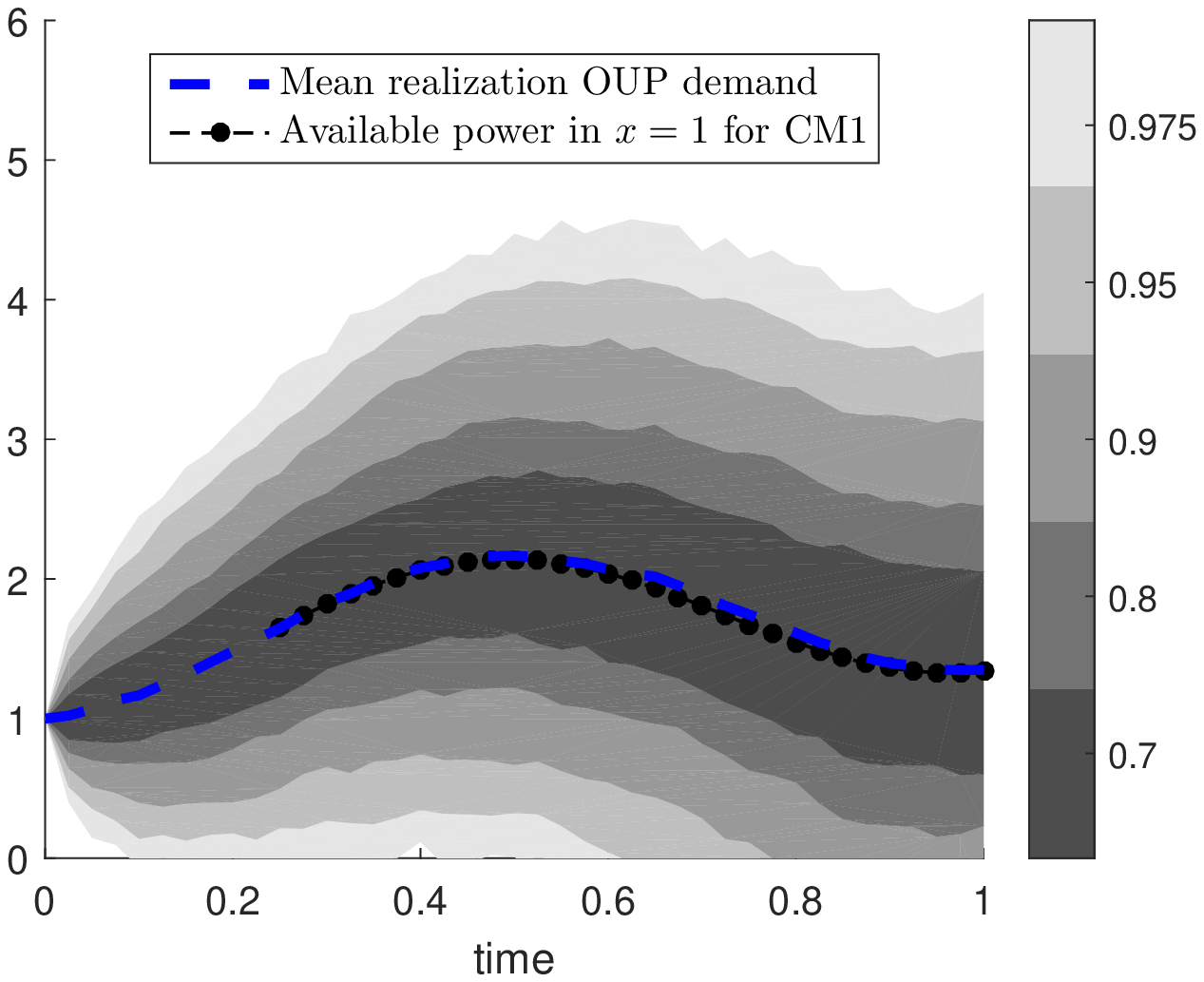}}\hfill
    \subfloat[\ $PS2$ \label{fig:outputOUPnoUpdates_lambda4_initialFix1_2Plus3Malsin2pit_kappa3_tildeSigma2_MC10P3}]{\includegraphics[width=0.5\linewidth]{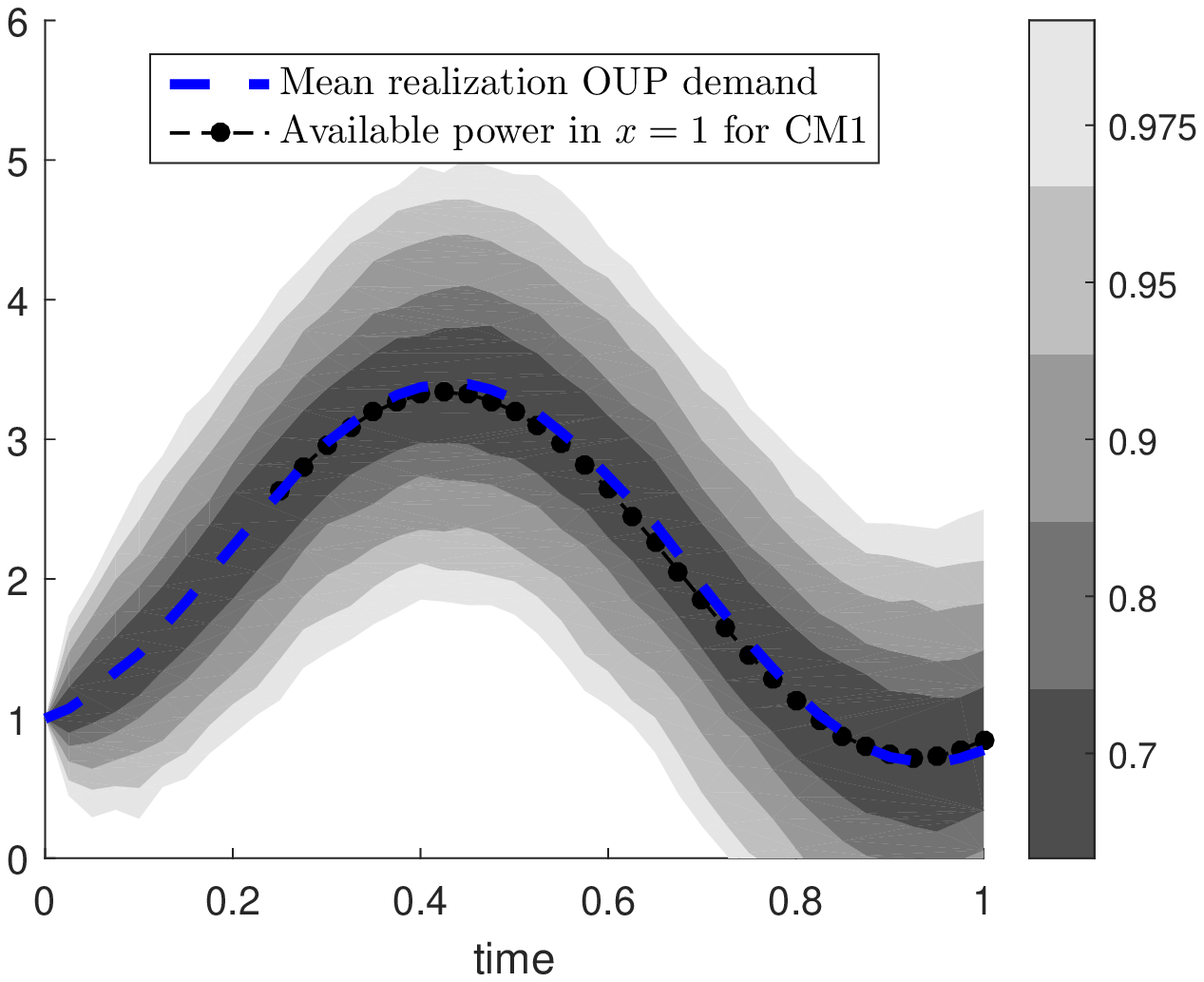}}
	\caption{Available power in $x=1$, mean realization and confidence levels of demand}
\end{figure}

Now, we address the numerical solution of \eqref{eq:SOC-Prob_update} in the setting of regular market updates. The update frequency is chosen to be $\Delta t^{\text{up}}=5$.
The optimal control (black with asterisk), output (red-dotted), and the updated mean realization of the demand (blue-dashed) corresponding to $PS1$ are depicted in Figure \ref{fig:transportlambda4initialfix12plus3malsin2pitkappa1tildesigma2mc10p3innermc10p3path5}. In addition, the linear transport is visualized by the straight black arrows from the control to the output. We can see that the optimal control is close to the updated mean stochastic demand shifted by the transportation time $\invLa = 0.25$. 

\begin{figure}[h!]
\subfloat[\ Optimal control, output and updated mean realization of demand \label{fig:transportlambda4initialfix12plus3malsin2pitkappa1tildesigma2mc10p3innermc10p3path5}]{\includegraphics[width=0.5\linewidth]{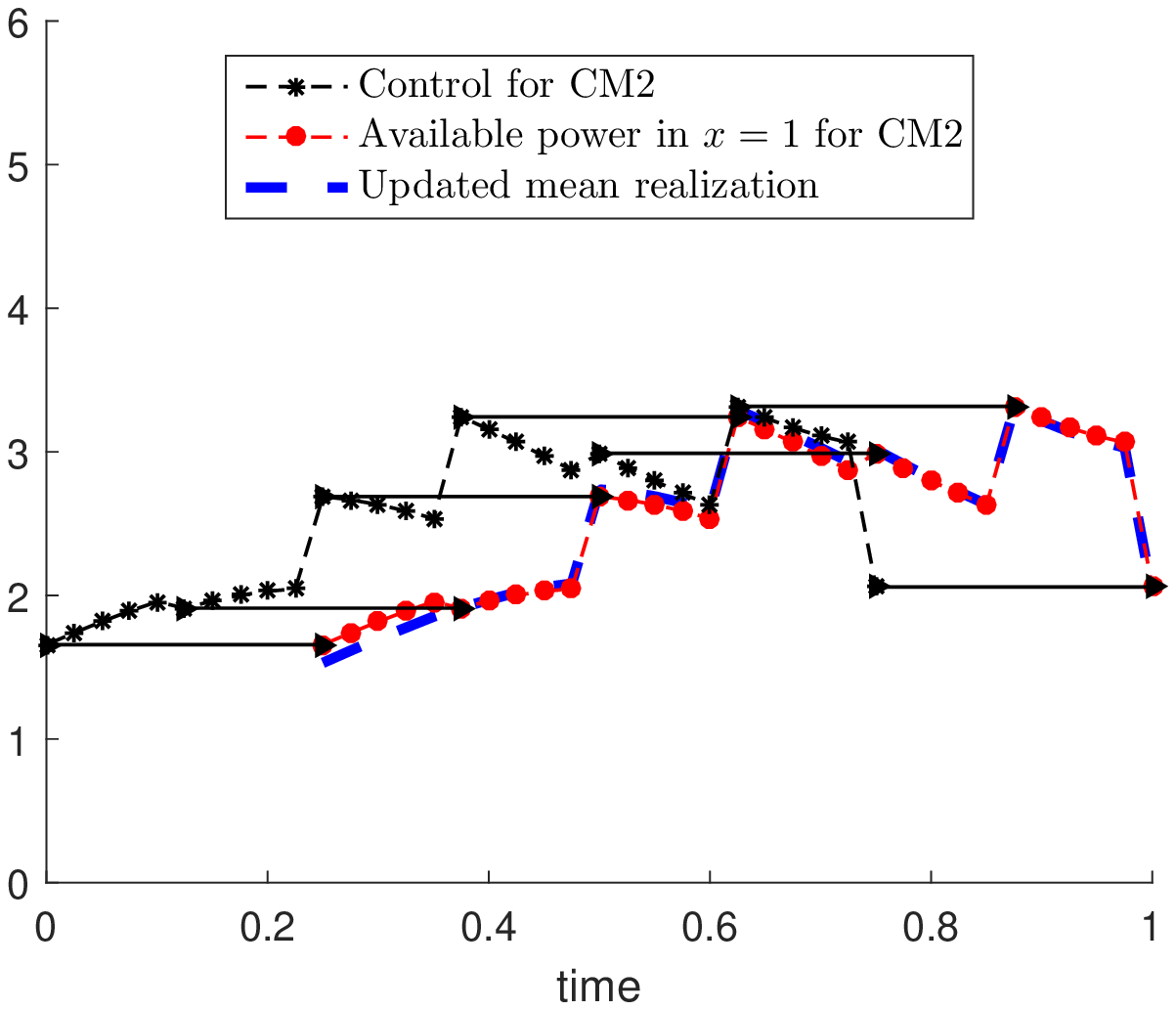}}\hfill
\subfloat[\ Comparison of the output with and without updates  \label{fig:outputstochasticslambda4initialfix12plus3malsin2pitkappa1tildesigma2mc10p3innermc10p3path5}]{\includegraphics[width=0.5\linewidth]{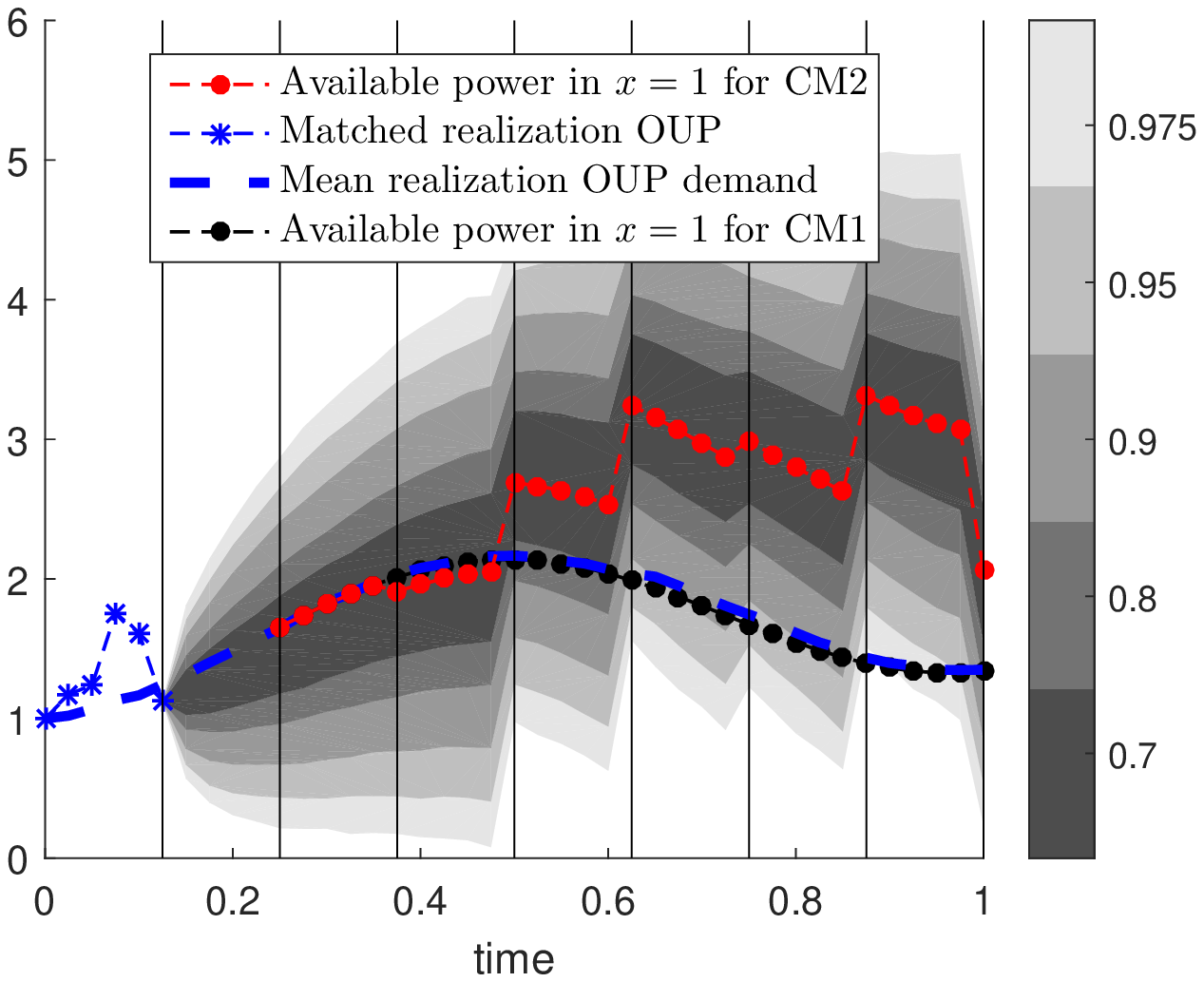}}
\caption{Numerical results for $PS1$ based on the update algorithm}
\end{figure}

Next, we compare the performance of the setting without updates (CM1) to the one with regular updates (CM2) for parameter setting $PS1$. The focus is on the quality of results obtained by CM1 and CM2.
 In Figure \ref{fig:outputstochasticslambda4initialfix12plus3malsin2pitkappa1tildesigma2mc10p3innermc10p3path5}, we plot the resulting available power from the optimal control without updates (black-dotted), and with updates (red-dotted), the mean realization of the OUP-demand (blue-dashed), the confidence levels to the updated forecasts of the demand in grey scale, and the path of the demand process, where the forecast is based on, as a blue-dashed line with asterisks up to the first update time. It results that, on average, the output of CM2 outperforms the one of CM1.

This outperformance can also be measured by the so-called \textit{cumulative root mean squared error (cumRMSE)}, which we define as
\begin{align*}
	\text{cumRMSE}(y(t)) \defgl \int_{\Delta t}^{T} \sqrt{\E\left[(Y_t - y(t))^2\right]}.
\end{align*}
The numerical cumRMSE of CM1 and CM2 can be found in Table \ref{tab:Comparison_cumRMSE}.
\begin{table}[h!]
	\centering
	\begin{tabular}{|r|rr|r|}
		\hline
		& CM1 & CM2 & relative reduction\\
		\hline
		PS1 & $0.9325$ & $0.7526$ & 19.29\%\\
		PS2 & $0.6434$ & $0.6002$ & 6.71\% \\
		\hline
	\end{tabular}%
	\caption{Comparison of cumRMSEs with and without updates for $PS1$ and $PS2$}
	\label{tab:Comparison_cumRMSE}
\end{table}

We observe that we really attain a relative reduction of the cumRMSE passing from CM1 to CM2 in both parameter settings. For $PS1$, where the speed of mean reversion is lower, the relative reduction is even more pronounced, which goes along with our intuitive understanding: Having a lower attraction to the mean demand level, leads to a slower mean reversion and, in tendency, to an increase in the time the process passes away from the mean demand level. Hence, the updates of the demand actually realized gain importance.

\subsubsection*{JDP-type demand: Parameter setting PS3}
Finally, we present the numerical results for a JDP-type demand using parameter setting $PS3$. Note that, except for the jumps, this setting is equal to $PS2$. Again, we start with strategy CM1, i.e without updates on the demand. 
The first important observation is that the optimization still works in the presence of jumps: The optimal control shifted by the transportation time matches well the mean realization of the demand, see Figure \ref{fig:transportJDPNoUpdates_lambda4_initialFix1_2Plus3Malsin2pit_kappa3_tildeSigma2_MC10P3_nu5_gamma1}. The second observation shows a clear difference between the OUP-type and the JDP-type demand. 
Whereas in the OUP-setting, the mean realization of the process lies slightly below 1 around $T=1$ (cf. Figure \ref{fig:transportNoUpdates_lambda4_initialFix1_2Plus3Malsin2pit_kappa3_tildeSigma2_MC10P3}), the mean realization in the JDP-setting at the same time lies slightly above $2$ (cf. Figure \ref{fig:transportJDPNoUpdates_lambda4_initialFix1_2Plus3Malsin2pit_kappa3_tildeSigma2_MC10P3_nu5_gamma1}). This shows a clear upward trend of the JDP-type demand which is due to the positive fixed jump height $\gamma \equiv 1$. This nonzero jump height also leads to an increase in the amplitude of the confidence intervals (compare Figures \ref{fig:outputOUPnoUpdates_lambda4_initialFix1_2Plus3Malsin2pit_kappa3_tildeSigma2_MC10P3} and \ref{fig:outputJDPnoUpdates_lambda4_initialFix1_2Plus3Malsin2pit_kappa3_sigma2_MC10P3_nu5_gamma1}).
\begin{figure}[h!]
	\subfloat[\ Optimal control, output and mean realization of demand \label{fig:transportJDPNoUpdates_lambda4_initialFix1_2Plus3Malsin2pit_kappa3_tildeSigma2_MC10P3_nu5_gamma1}]{\includegraphics[width=0.5\linewidth]{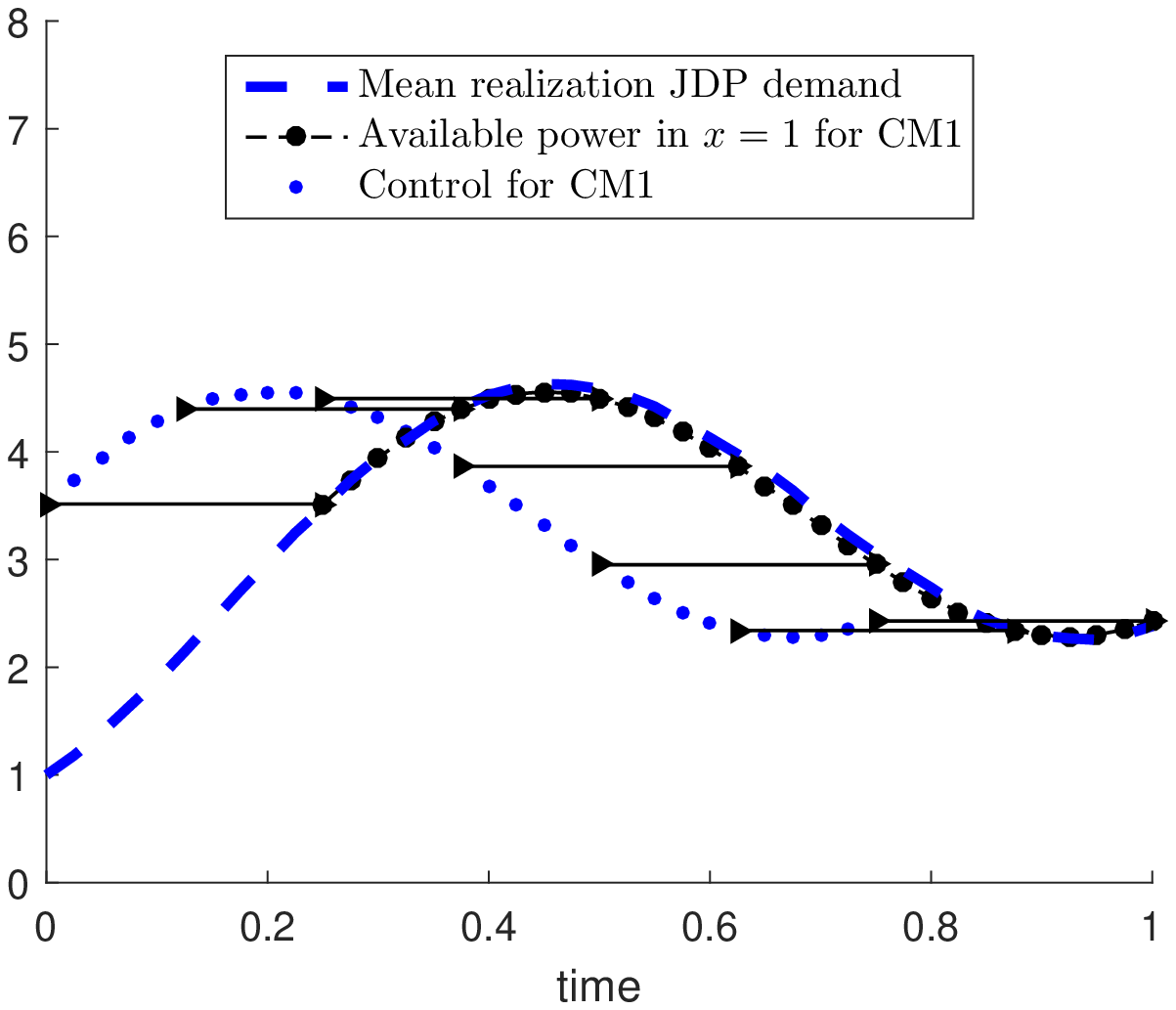}}\hfill
	\subfloat[\ Available power in $x=1$, mean realization and confidence levels of demand \label{fig:outputJDPnoUpdates_lambda4_initialFix1_2Plus3Malsin2pit_kappa3_sigma2_MC10P3_nu5_gamma1}]{\includegraphics[width=0.5\linewidth]{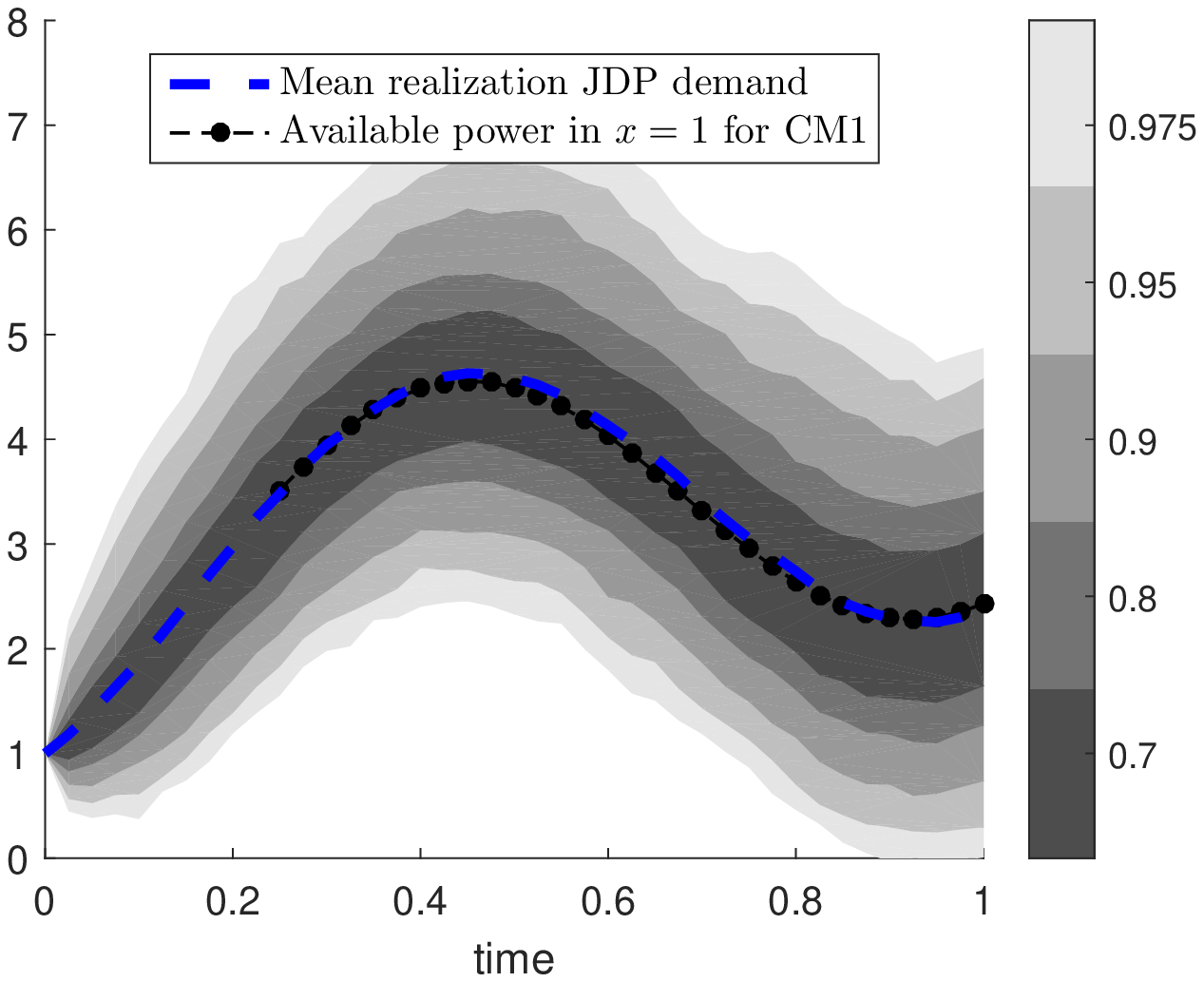}}\hfill
    
	\caption{Numerical results for $PS3$ without updates}
\end{figure}

\begin{figure}[h!]
	\subfloat[\ Optimal control, output, and updated mean realization of demand  \label{fig:transportJDPwithUp_lambda4_initial1_2plus3sin2pit_kappa3_sigma2_nu5_gamma1_T1_up5_MC1000path5}]{\includegraphics[width=0.5\linewidth]{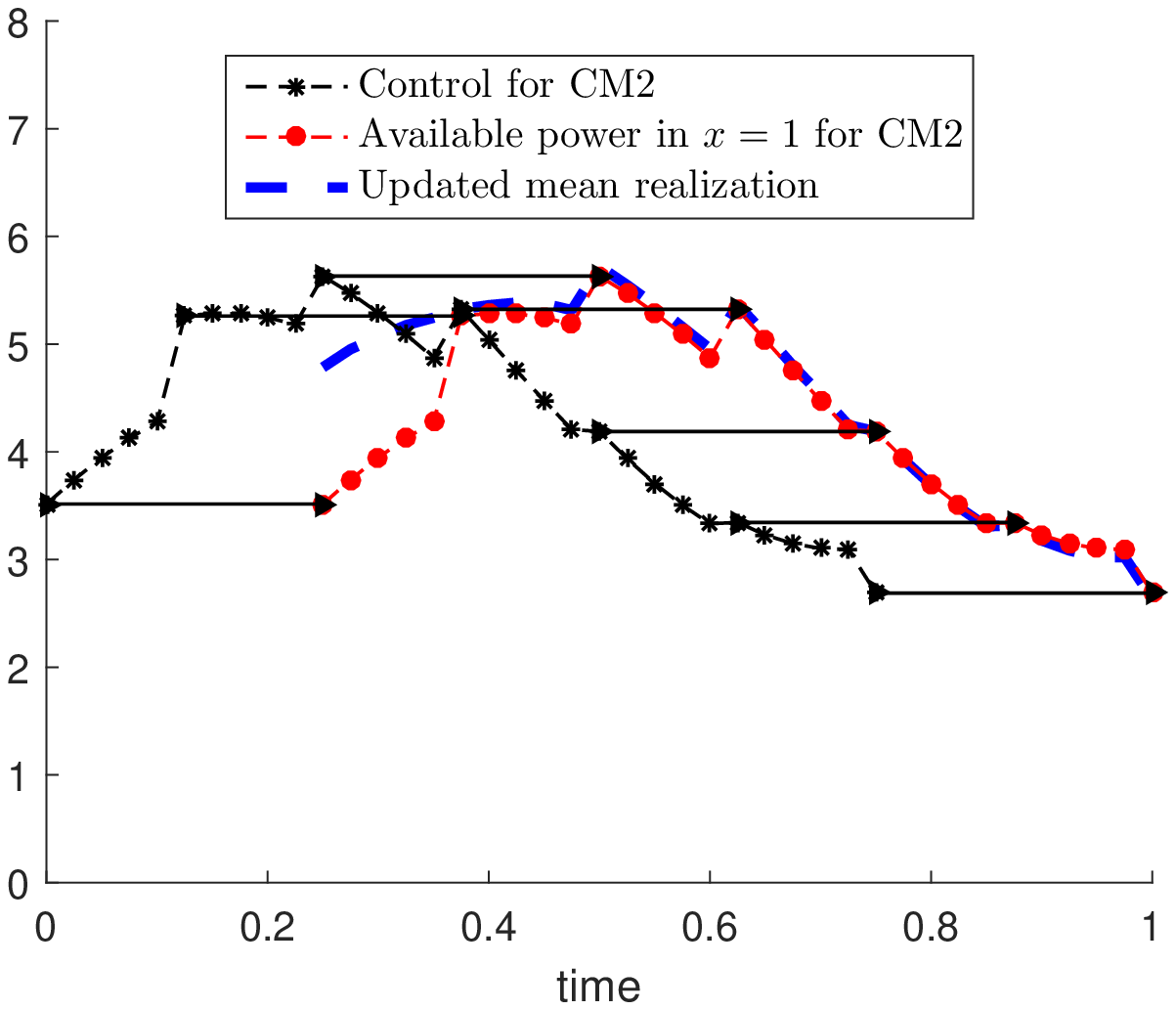}}
    \subfloat[\ Comparison of the output with and without updates  \label{fig:outputJDPwithUp_lambda4_initial1_2plus3sin2pit_kappa3_sigma2_nu5_gamma1_T1_up5_MC1000path5}]{\includegraphics[width=0.5\linewidth]{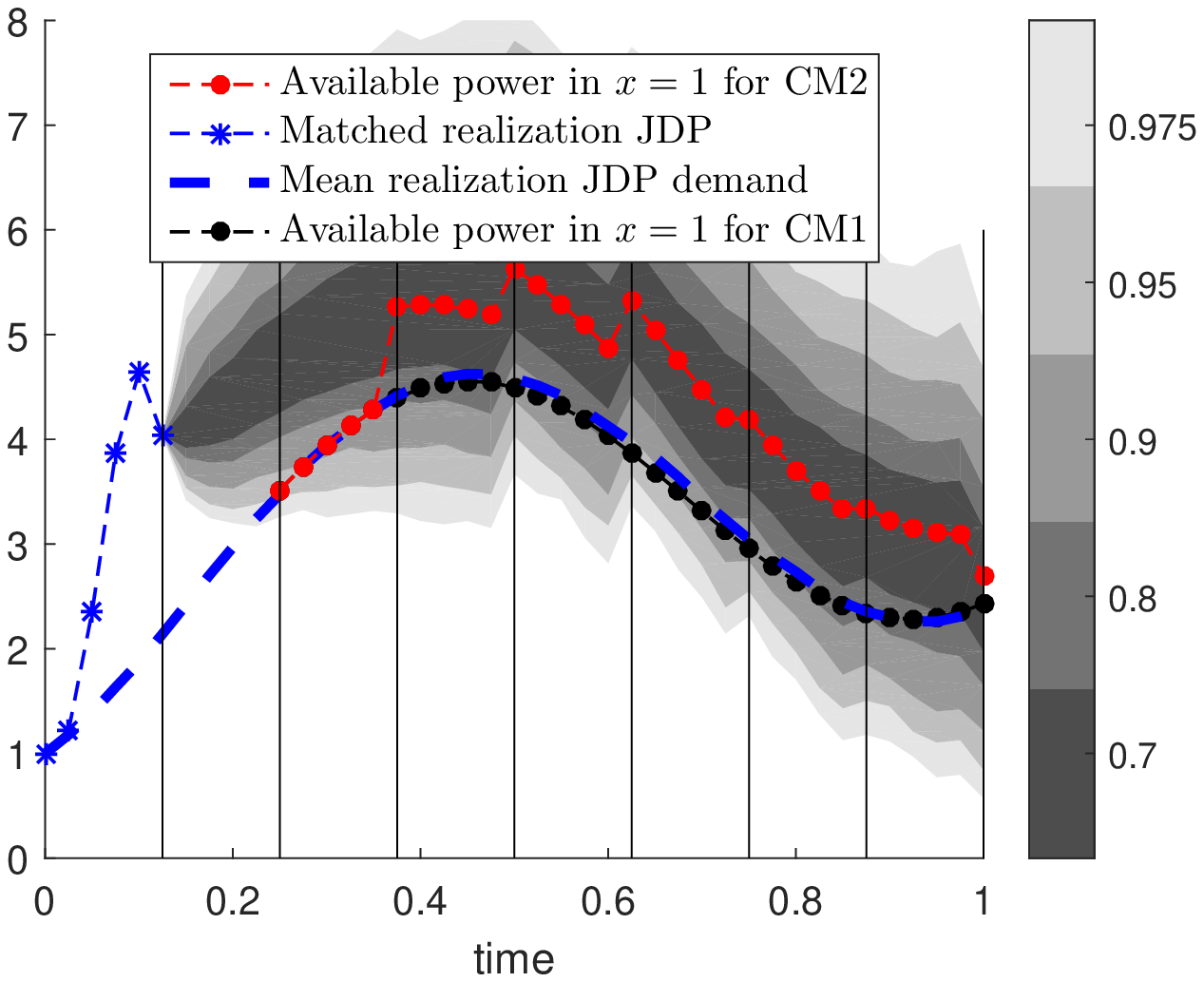}}
	\caption{Numerical results for $PS3$ based on update algorithm}
    \label{fig:numResPS3update}
\end{figure}
In Figure \ref{fig:numResPS3update}, we see that, on average, updates help to enhance the performance especially because they better capture the upward trend of the demand, which is due to the fixed positive jump height.

\bigskip

In order to see how Theorem \ref{thm:convergenceUpToTheo} applies to our numerical example, we analyze the convergence behavior for decreasing time instances between updates for parameter setting PS3. To be more precise, we calculate the cumRMSE between the numerical solution obtained by applying CM2 and the numerical implementation of the theoretical solution \eqref{eq:optControlJDP} based on CM3.

In Table \ref{tab:convTheoreticalSol}, we observe that the cumRMSE decreases with decreasing time instances between updates and tends to zero. This coincides with the analytical result in Theorem \ref{thm:convergenceUpToTheo}.
\begin{table}[h!]
	\centering
	\begin{tabular}{cc}
		\hline 
		Time instances between updates & cumRMSE \\ 
		\hline 
		41 & 0.3885 \\ 
		5 &  0.1924 \\ 
		3 &  0.1278 \\
		2 &  0.0814 \\
		1 &  2.6701e-06 \\ 
		\hline 
	\end{tabular}
	\caption{Convergence of numerical solution based on CM2 against numerical implementation of theoretical solution for CM3}
	\label{tab:convTheoreticalSol}
\end{table}


\section{Conclusion}
In this paper, we derived an optimal control strategy for electricity production under three different information scenarios. The underlying stochastic optimal control problem was analyzed from a theoretical and numerical point of view. A suitable reformulation of the problem also allowed for the application of standard optimization solvers. 
The theoretical results were supplemented by numerical simulations illustrating the relation and convergence behavior of solutions based on the different control methods.

Future work includes the investigation of more involved (nonlinear) dynamics to describe the transport along the system, e.g.\ Euler equations for gas transport. This might lead to an appropriate Fokker-Planck control framework since the control cannot be computed in a direct way anymore. However, an adjoint calculus can be applied to study first order optimality conditions.

\section{Appendix} \label{app:secondMomJDP}
The detailed calculation of the second moment of the JDP used in Section \ref{subsec:analyticalTreatmentCost} is as follows:
\begin{align*}
\E\left[Y_t^2\right] =& \E\left[\left(e^{-\kappa t}y_0 + \sigma \int_{0}^{t} e^{-\kappa (t-s)} dW_s + \kappa \int_{0}^{t} e^{-\kappa (t-s)}\mu(s) ds +  \sum_{i=1}^{N_t} \gamma_{t_i} e^{-\kappa (t-t_i)}\right)^2\right] \notag\\
=& \E\left[\left(e^{-\kappa t}y_0 + \kappa \int_{0}^{t} e^{-\kappa (t-s)}\mu(s) ds\right)^2 + \left(\sigma \int_{0}^{t} e^{-\kappa (t-s)} dW_s + \sum_{i=1}^{N_t} \gamma_{t_i} e^{-\kappa (t-t_i)}\right)^2\right.\notag\\
&\left. +2\left(e^{-\kappa t}y_0 + \kappa \int_{0}^{t} e^{-\kappa (t-s)}\mu(s) ds\right) \cdot \left(\sigma \int_{0}^{t} e^{-\kappa (t-s)} dW_s + \sum_{i=1}^{N_t} \gamma_{t_i} e^{-\kappa (t-t_i)}\right) \right]
\end{align*}
\begin{align*}
=& \left(e^{-\kappa t}y_0 + \kappa \int_{0}^{t} e^{-\kappa (t-s)}\mu(s) ds\right)^2\E\left[\left(\sigma \int_{0}^{t} e^{-\kappa (t-s)} dW_s\right)^2\right] +\E\left[\left(\sum_{i=1}^{N_t}\gamma_{t_i} e^{-\kappa (t-t_i)}\right)^2\right] \notag\\
&+ 2\E\left[\sigma \int_{0}^{t} e^{-\kappa (t-s)} dW_s\sum_{i=1}^{N_t} \gamma_{t_i} e^{-\kappa (t-t_i)}\right] +2\left(e^{-\kappa t}y_0 + \kappa \int_{0}^{t} e^{-\kappa (t-s)}\mu(s) ds\right)\notag\\
& \cdot \E\left[\sigma \int_{0}^{t} e^{-\kappa (t-s)} dW_s + \sum_{i=1}^{N_t} \gamma_{t_i} e^{-\kappa (t-t_i)}\right]\notag\\
&= \left(e^{-\kappa t}y_0 + \kappa \int_{0}^{t} e^{-\kappa (t-s)}\mu(s) ds\right)^2 + \E\left[\sigma^2 \int_{0}^{t} e^{-2 \kappa (t-s)} ds\right] +\E\left[\left(\sum_{i=1}^{N_t}\gamma_{t_i} e^{-\kappa (t-t_i)}\right)^2\right] \notag\\
&+ 2\E\left[\sigma \int_{0}^{t} e^{-\kappa (t-s)} dW_s\right]\E\left[\sum_{i=1}^{N_t} \gamma_{t_i} e^{-\kappa (t-t_i)}\right]\notag\\
&+2\left(e^{-\kappa t}y_0 + \kappa \int_{0}^{t} e^{-\kappa (t-s)}\mu(s) ds\right) \cdot \E\left[\sigma \int_{0}^{t} e^{-\kappa (t-s)} dW_s\right]\notag\\
&+ 2\left(e^{-\kappa t}y_0 + \kappa \int_{0}^{t} e^{-\kappa (t-s)}\mu(s) ds\right) \cdot \E\left[\sum_{i=1}^{N_t}\gamma_{t_i} e^{-\kappa (t-t_i)}\right]\notag\\
&= \left(e^{-\kappa t}y_0 + \kappa \int_{0}^{t} e^{-\kappa (t-s)}\mu(s) ds\right)^2 + \sigma^2 \left[ \frac{e^{-2 \kappa (t-s)}}{2\kappa}\right]_{s=0}^{s=t} + \E\left[\left(\sum_{i=1}^{N_t}\gamma_{t_i} e^{-\kappa (t-t_i)}\right)^2\right] \notag\\
&+ 2\left(e^{-\kappa t}y_0 + \kappa \int_{0}^{t} e^{-\kappa (t-s)}\mu(s) ds\right) \cdot \E\left[\sum_{i=1}^{N_t}\gamma_{t_i} e^{-\kappa (t-t_i)}\right]\notag\\
&= y_0^2e^{-2\kappa t} + 2y_0e^{-\kappa t} \int_{0}^{t}e^{-\kappa(t-s)}\kappa \mu(s) ds + \left(\kappa\int_{0}^{t}e^{-\kappa(t-s)} \mu(s) ds\right)^2 + \frac{\sigma^2}{2\kappa}\left(1-e^{-2 \kappa t}\right) \notag\\
&+ \E\left[\left(\sum_{i=1}^{N_t}\gamma_{t_i} e^{-\kappa (t-t_i)}\right)^2\right] + 2\left(e^{-\kappa t}y_0 + \kappa \int_{0}^{t} e^{-\kappa (t-s)}\mu(s) ds\right) \cdot \E\left[\gamma_{t_i}\right] \frac{\nu}{\kappa}\left(1-e^{-\kappa t}\right). 
\end{align*}
Note that the second moment of the time-dependent OUP is obtained by setting $\gamma_{t_i} \equiv 0$ for all jump times $t_i$.
\bigskip
Thus, it remains to calculate $\E\left[\left(\sum_{i=1}^{N_t}\gamma_{t_i} e^{-\kappa(t-t_i)}\right)^2\right]$.
\begin{align*}
	\E\left[\left(\sum_{i=1}^{N_t}\gamma_{t_i} e^{-\kappa(t-t_i)}\right)^2\right]
    =& \E\left[\sum_{i=1}^{N_t}\gamma_{t_i}^2 e^{-2\kappa(t-t_i)} + \sum_{i\neq j}^{N_t}\gamma_{t_i} e^{-\kappa(t-t_i)}\gamma_{t_j} e^{-\kappa(t-t_j)}\right] \\
    =& \E\left[\sum_{i=1}^{N_t} e^{-2\kappa(t-t_i)}\right]\E\left[\gamma^2\right] + \E\left[\sum_{i\neq j}^{N_t} e^{-\kappa(t-t_i)}e^{-\kappa(t-t_j)}\right]\E\left[\gamma\right]^2\\
    =& \E\left[\sum_{i=1}^{N_t} e^{-2\kappa(t-t_i)}\right] \E\left[\gamma^2\right] \\
    & + 2\cdot \E\left[\E\left[\sum_{i=2}^{N_t} \sum_{j=1}^{i-1} e^{-\kappa(t-t_i)}e^{-\kappa(t-t_j)}|N_t\right]\right]\E\left[\gamma\right]^2.
\end{align*}
We know from \cite[Prop.\ 2.\ 1.\ 16]{Mikosch.2009} that $t_i \sim \UU[0,t]$. Thus, we calculate
\begin{align*}
	\E\left[e^{-2\kappa(t-t_i)}\right] &= \int_{0}^{t}e^{-2\kappa(t-s)}\frac{1}{t}ds = \frac{1-e^{-2\kappa t}}{2\kappa t}.
\end{align*}
We can then deduce
\begin{align*}
	\E\left[\sum_{i=1}^{N_t} e^{-2\kappa(t-t_i)}\right] &= \sum_{i=1}^{\infty} i \cdot e^{-\nu t} \frac{(\nu t)^{i}}{i !} \frac{1-e^{-2\kappa t}}{2\kappa t} = \nu \cdot \frac{(1-e^{-2\kappa t})}{2\kappa}.
\end{align*}
It remains to compute the mixed-term expectation.
\begin{align*}
	\E\left[e^{-\kappa(t-t_i)}e^{-\kappa(t-t_j)}\right | t_j<t_i, N_t=n] &= \int_{0}^{t}e^{-\kappa(t-s)} \frac{1}{t} \int_{0}^{s} e^{-\kappa(t-u)}\frac{2}{t} du ds \\
    &= \int_{0}^{t} \frac{2\cdot(e^{-2\kappa(t-s)}-e^{-\kappa (2t-s)})}{\kappa t^2} ds \\
    &= \frac{1+e^{-2\kappa t}-2e^{-\kappa t}}{\kappa^2 t^2}.
\end{align*}
Thus, we have
\begin{align*}
	\E\left[\sum_{i=2}^{N_t} \sum_{j=1}^{i-1} e^{-\kappa(t-t_i)}e^{-\kappa(t-t_j)}|N_t=n\right] =& \sum_{i=2}^{n} \sum_{j=1}^{i-1} \E\left[e^{-\kappa(t-t_i)}e^{-\kappa(t-t_j)}|N_t=n\right] \\
    =& \sum_{i=1}^{n-1} i \cdot \E\left[e^{-\kappa(t-t_{i+1})}e^{-\kappa(t-t_1)}|N_t=n\right] \\
    =&  \E\left[e^{-\kappa(t-t_{2})}e^{-\kappa(t-t_1)}|N_t=n\right] \cdot \frac{n \cdot (n-1)}{2} \\
    =& \frac{1+e^{-2\kappa t}-2e^{-\kappa t}}{\kappa^2 t^2} \cdot \frac{n \cdot (n-1)}{2}.
\end{align*}
\begin{align*}
	\E\left[\E\left[\sum_{i=2}^{N_t} \sum_{j=1}^{i-1} e^{-\kappa(t-t_i)}e^{-\kappa(t-t_j)}|N_t\right]\right] &= \E\left[N_t^2-N_t\right] \cdot \frac{1+e^{-2\kappa t}-2e^{-\kappa t}}{2 \kappa^2 t^2} \\
    &= \nu^2 \cdot \frac{1+e^{-2\kappa t}-2e^{-\kappa t}}{2 \kappa^2}.
\end{align*}
Finally, the closed-form expression is
\begin{align*}
	\E\left[\left(\sum_{i=1}^{N_t}\gamma_{t_i} e^{-\kappa(t-t_i)}\right)^2\right] &= \nu \cdot \frac{(1-e^{-2\kappa t})}{2\kappa} \cdot \E\left[\gamma^2\right] + \nu^2 \cdot \frac{1+e^{-2\kappa t}-2e^{-\kappa t}}{\kappa^2} \cdot \E\left[\gamma\right]^2.
\end{align*}

\bibliographystyle{siam}
\bibliography{mysources}

\end{document}